\documentclass[11pt]{article}
\usepackage[utf8]{inputenc}
\usepackage[top=1in, bottom=1in, left=1in, right=1in]{geometry}
\usepackage{authblk}

\usepackage{amsmath,amsfonts,amssymb,graphicx,csquotes,mathtools}
\usepackage{mathrsfs}
\usepackage{enumerate}

\usepackage{fixltx2e}

\usepackage[dvipsnames]{xcolor} 
\usepackage{hyperref} 
\hypersetup{
    colorlinks=true,
    linkcolor=MidnightBlue,
    citecolor=blue,
    urlcolor=blue,
    pdftitle={Doubly Intermittent Full Branch Maps \\ with Critical Points and Singularities},
    }
\usepackage{color}
\usepackage[dvipsnames]{xcolor}
\usepackage{tikz-cd}

\usepackage{graphicx}
\graphicspath{plots}
\usepackage{subcaption}

\usepackage[
backend=bibtex,
style=alphabetic,
sorting=nyt,
maxbibnames=99
]{biblatex}
\makeatletter
\def\namedlabel#1#2{\begingroup
    #2%
    \def\@currentlabel{#2}%
    \phantomsection\label{#1}\endgroup
}
\makeatother

\usepackage[colorinlistoftodos]{todonotes}

\usepackage{amsthm}

\theoremstyle{definition}
\newtheorem{defn}{Definition}[section]

\theoremstyle{plain}
\newtheorem{lem}[defn]{Lemma}

\newtheorem*{thm*}{Theorem}
\newtheorem{prop}[defn]{Proposition}
\theoremstyle{plain}
\newtheorem{cor}[defn]{Corollary}

\theoremstyle{remark}
\newtheorem{rem}[defn]{Remark}
\theoremstyle{definition}

\makeatother

\numberwithin{defn}{section}

\newtheorem{manualtheoreminner}{Theorem}
\newenvironment{manualtheorem}[1]{%
  \renewcommand\themanualtheoreminner{#1}%
  \manualtheoreminner
}{\endmanualtheoreminner}

\newcommand{\eps}{\varepsilon}
\newcommand{\quand}{\quad\text{ and } \quad}

\title{Doubly Intermittent Full Branch Maps \\ with Critical Points and Singularities}

\author[3]{Douglas Coates\footnote{\href{mailto:dc485@exeter.ac.uk}{\texttt{dc485@exeter.ac.uk}}}}

\author[1]{Stefano Luzzatto \footnote{\href{mailto:luzzatto@ictp.it}{\texttt{luzzatto@ictp.it}}}} 

\author[1,2]{
Mubarak Muhammad
\footnote{\href{mailto:mmuhamma@ictp.it}{\texttt{mmuhamma@ictp.it}}, \href{mailto:mubarak@aims.edu.gh}{\texttt{mubarak@aims.edu.gh}}}}

\affil[1]{Abdus Salam International Centre for Theoretical Physics (ICTP), Trieste, Italy}

\affil[2]{Scuola Internazionale Superiorie di Studi Avanzati (SISSA), Trieste, Italy}

\affil[3]{University of Exeter, Exeter, UK}

\date{26 September 2022}

\begin{document}
\maketitle
\begin{abstract}
We study a class of one-dimensional full branch maps admitting two indifferent fixed points as well as critical points and/or unbounded derivative.  Under some mild assumptions we prove the existence of a unique invariant mixing absolutely continuous probability measure,  study its rate of  decay of correlation and prove a number of limit theorems. 
\end{abstract}

\setcounter{tocdepth}{2}

\section{Introduction}

The purpose of this paper is to study the ergodic properties of a large class of full branch interval maps with two branches, including maps with 
\emph{two}  indifferent fixed points (which, as we shall see below, affects both the results and the construction of the induced map which we require). We also allow the derivative to go to \emph{zero} as well as to \emph{infinity} at the boundary between the two branches, and we do not assume any symmetry, even the domains of the branches can be of arbitrary length.
Such maps are known to exhibit a wide range of behaviour from an ergodic point of view and many of them have been extensively studied, we give a detailed literature review below. 
 
\begin{figure}[h]
 \centering
  \begin{subfigure}{.33\textwidth}
    \centering
    \includegraphics[width=6cm]{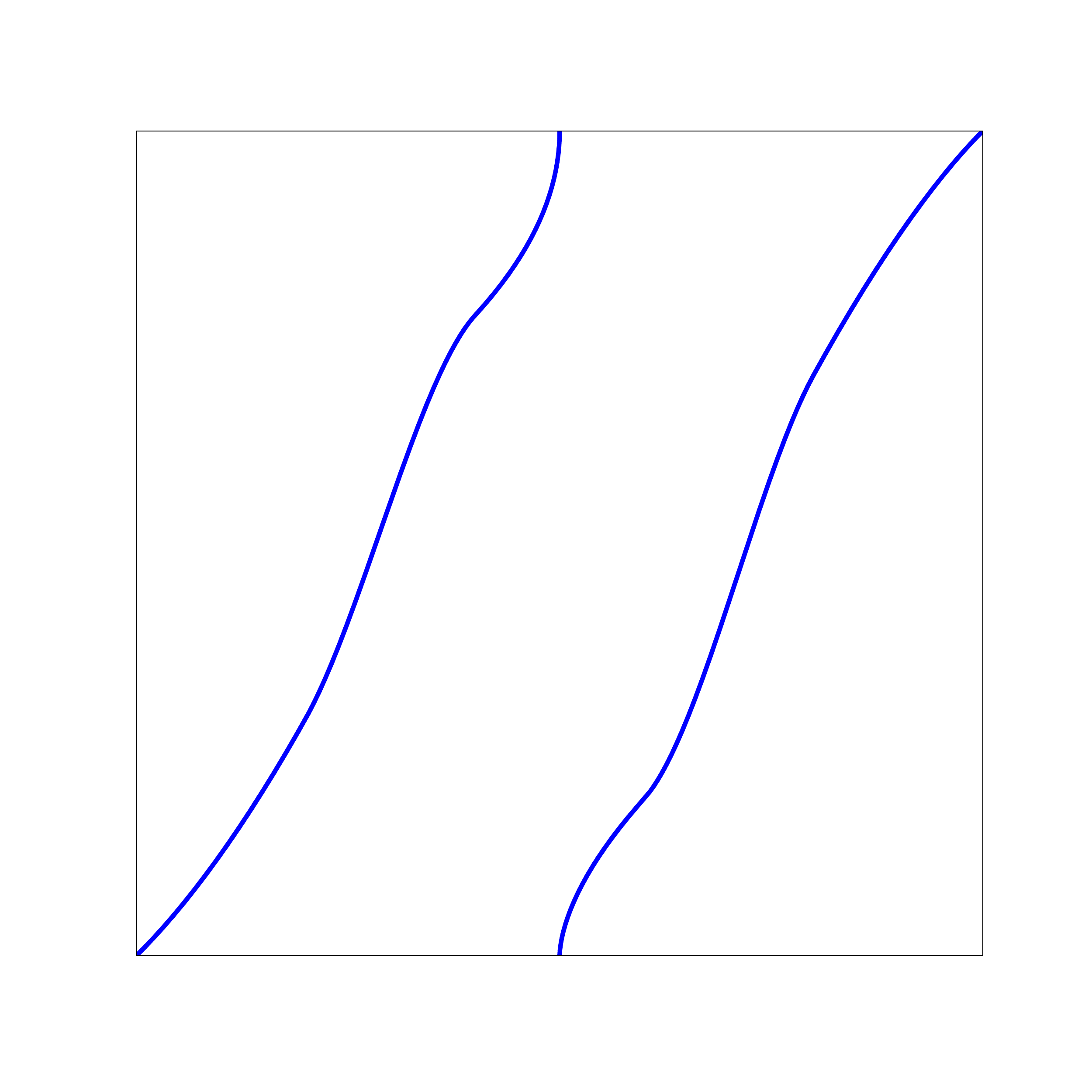}
  \end{subfigure}%
  \hspace{-1em}
  \begin{subfigure}{.33\textwidth}
    \centering
    \includegraphics[width=6cm]{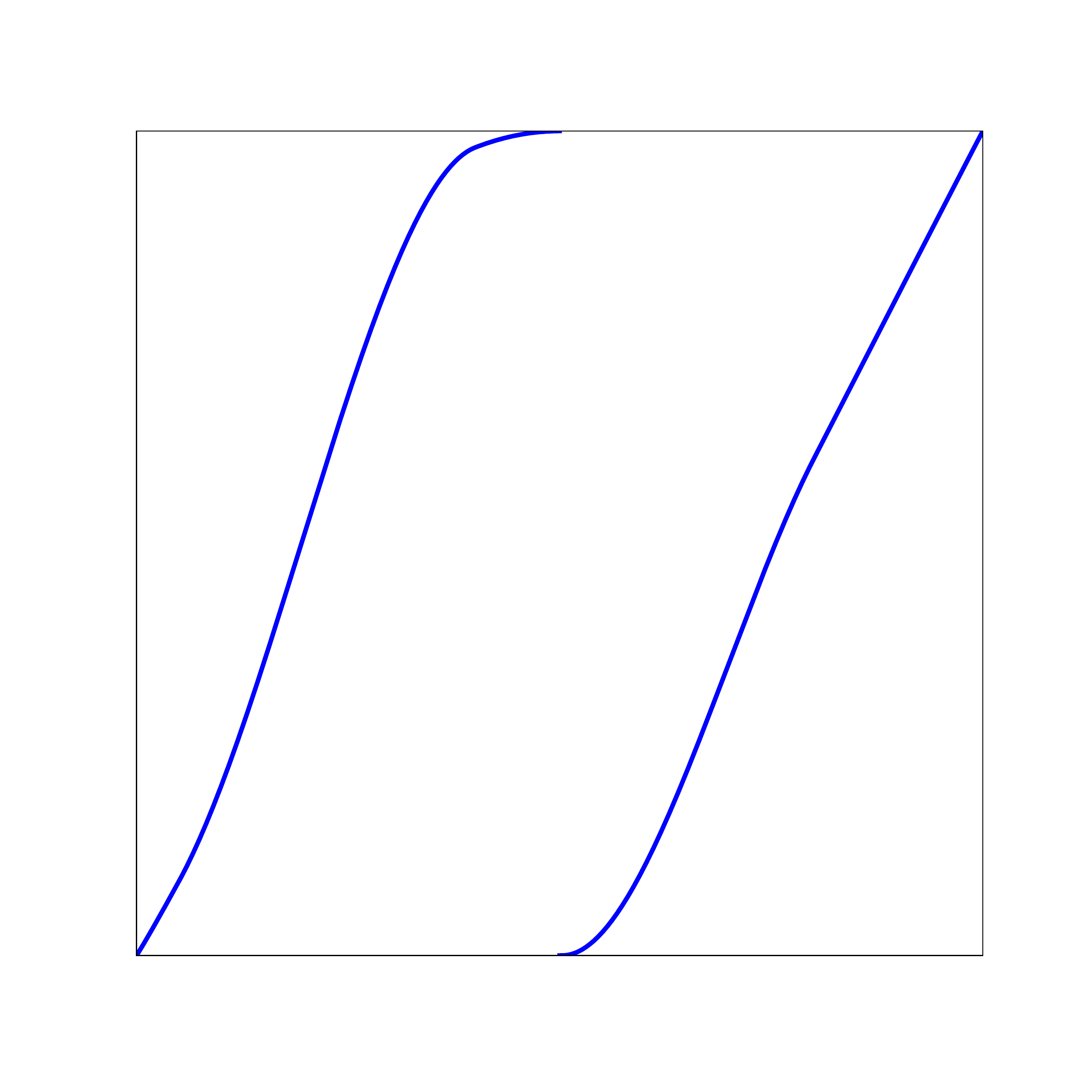}
  \end{subfigure}
   \hspace{-1em}
  \begin{subfigure}{.33\textwidth}
    \centering
    \includegraphics[width=6cm]{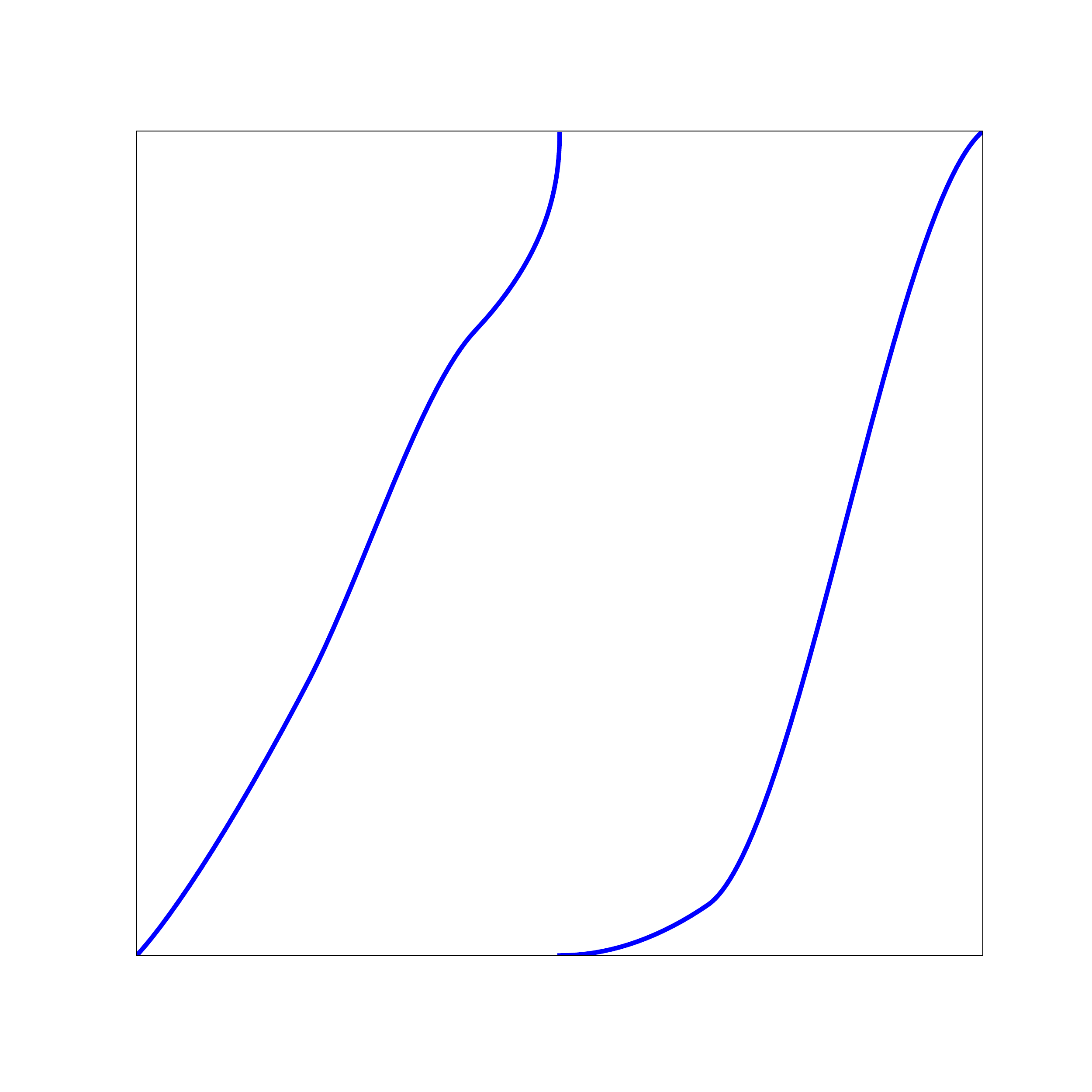}
  \end{subfigure}
  \caption{Graph of $g$ for various possible values of parameters.}\label{fig:fig}
\end{figure}

In Section~\ref{sec:full branch} we give the precise definition of the class of maps  we consider, which includes many cases already studied in the literature as well as many cases which have not yet been studied; in section~\ref{sec:results} we give the precise statements of our results; in section \ref{sec:lit-review} we give a literature review of related results and include specific examples of maps in our family; in Section \ref{sec:outline} we give a detailed outline of our proof, emphasising several novel aspects of our construction and arguments. Then in Section \ref{sec:induced} we give the construction and estimates related to our ``double-induced'' map and in Section \ref{sec:statistical} apply these estimates to complete the proofs of our results

\subsection{Full Branch Maps}
\label{sec:full branch}

We start by  defining the class of maps which we consider in this paper.   Let \( I, I_-, I_+\) be compact intervals, let  \( \mathring I, \mathring I_-, \mathring I_+\) denote their interiors, and suppose that  \(I = I_{-}\cup I_{+}  \) and \( \mathring I_-\cap \mathring I_+=\emptyset\). 

\begin{description}
  \item{\namedlabel{itm:A0}{\textbf{(A0)}}} 
\( g: I \to I \) is \emph{full branch}: the restrictions \(g_{-}: \mathring I_{-}\to \mathring I\) and \(g_{+}: \mathring I_{+}\to \mathring I\) are  orientation preserving \( C^{2} \) diffeomorphisms and the only  fixed points are the endpoints of \( I \). 
\end{description}

To simplify the notation we will assume that  
\[ I=[-1, 1],  \quad 
I_{-}=[-1, 0],
\quad 
 I_{+}=[0,1]\]
  but our results and proofs will be easily seen to hold in the general setting.

\begin{description}
  \item{\namedlabel{itm:A1}{\textbf{(A1)}}} 
    There exists constants $\ell_1,\ell_2 \geq 0$, \( \iota, k_1,k_2 , a_1,a_2,b_1,b_2 > 0 \)  such that: 

    \begin{enumerate}[(i)]
      \item
        if \( \ell_{1}, \ell_{2}\neq 0 \) and \( k_{1}, k_{2} \neq  1\), then 
        \begin{equation}\label{eqn_1}
          g(x) =
          \begin{cases} 
            x+b_1{(1+x)}^{1+\ell_1} &  \text{in }  U_{-1},\\
            1-a_1{|x|}^{k_1} & \text{in } U_{0-}, \\
            -1+a_2{x}^{k_2} & \text{in }  U_{0+}, \\
            x-b_2{(1-x)}^{1+\ell_2} &  \text{in }  U_{+1},
          \end{cases}
        \end{equation}
          where
          \begin{equation}\label{eq:U}
            U_{0-}:=(-\iota, 0],
            \quad
            U_{0+}:=[0, \iota), 
            \quad
            U_{-1}:=g(U_{0+}), 
            \quad
            U_{+1}:=g(U_{0-}).
          \end{equation}
      \item
       If  \( \ell_{1}=0 \)  and/or \( \ell_2=0\) we replace the corresponding lines in~\eqref{eqn_1} with 
          \begin{equation}\label{eq:A1b}
            g|_{U_{\pm 1}}(x) \coloneqq \pm 1 + (1 + b_1) ( x + 1) \mp  \xi (x),
          \end{equation}
          where \( \xi \) is \( C^2\),  \(\xi(\pm 1)= 0, \xi'(\pm 1)=0\), and \( \xi''(x)>0\) on \( U_{-1}\) and  \( \xi''(x)<0\) on \( U_{+1} \).  \\
          If 
\( k_1 = 1 \) and/or \( k_2 = 1 \), then we replace the corresponding lines in~\eqref{eqn_1} with the assumption that  
\(
 g'(0_-) = a_1>1\) and/or \( g'(0_+) = a_2>1
 \)
   respectively, and that \( g \) is  monotone in the corresponding neighbourhood. 
    \end{enumerate}
\end{description}

\begin{rem}
It is easy to see that the definition in \eqref{eqn_1} yields maps with dramatically different derivative behaviour depending on the values of \( \ell_1, \ell_2, k_1, k_2\), including having neutral or expanding fixed points and points with zero or infinite derivative, see Remark \ref{rem:def} for a detailed discussion. For the moment we just remark that the assumptions described in part {ii)} of condition \ref{itm:A1} are  consistent with~\eqref{eqn_1} but significantly relax the definition given there as in these cases \eqref{eqn_1} would imply that  the map is affine in the corresponding neighbourhood, whereas we only need expansivity. In particular this allows us  to  include  uniformly expanding maps in our class of maps. In the calculations below we will explicitly consider the cases \( \ell_{1}=0 \)  and/or \( \ell_2=0\), which correspond to assuming that one or both the fixed points are expanding instead of neutral,  since they yield different estimates (several quantities decay exponentially rather than polynomially in these cases) and different results, and still include some maps which, as far as we know, have not been studied in the literature. For simplicity, on the other hand, we will not consider explicitly the cases \( k_1 = 1 \) and/or \( k_2 = 1 \), which just correspond to assuming the derivative at one or both sides of the discontinuity is finite instead of being zero or infinite.  These correspond to much simpler special cases and the required estimates follow by arguments  which are very similar to arguments and calculations we give here, and which are essentially already considered in the literature, but treating them explicitly would require a significant amount of additional notation and calculations.
\end{rem}

Our final assumption can be intuitively thought of as saying  that \( g \) is uniformly expanding outside the neighbourhoods \( U_{0\pm}\) and \( U_{\pm 1}\). This is however much stronger than what is needed, and therefore we formulate a weaker and more general assumption for which we need to describe some aspects   of the topological structure of maps satisfying condition \ref{itm:A0}. First of all we define 
\begin{equation}\label{eq:Delta0}
\Delta^-_0:= 
g^{-1}(0,1)\cap I_-
\quand 
\Delta^+_0:= 
 g^{-1}(-1,0)\cap I_+.  
\end{equation}
Then we define iteratively, for every \( n \geq 1 \),  the sets 
\begin{equation}\label{eq:Delta}
 \Delta_n^{-}:= g^{-1}(\Delta_{n-1}^{-})\cap I_{-}
 \quand 
 \Delta_n^{+}:= g^{-1}(\Delta_{n-1}^{+})\cap I_{+}
 \end{equation}
 as the \( n\)'th preimages of \( \Delta_0^-, \Delta_0^+\) inside the intervals \(I_{-}, I_{+} \).   It follows from \ref{itm:A0} that 
\( 
 \{ \Delta_n^{-}\}_{n\geq 0}
\) 
and \( 
 \{ \Delta_n^{+}\}_{n\geq 0}
\) 
are $\bmod\;0$ partitions of \(I_{-}\) and \(I_{+}\) respectively, and that the partition elements depend \emph{monotonically} on the index  in the sense that \( n > m \) implies that \( \Delta_n^{\pm}\) is closer to \( \pm 1\) than \( \Delta_m^{\pm}\),  in particular the only accumulation points of these partitions are \( -1\) and \( 1 \) respectively.  
Then, for every \( n \geq 1 \),  we let 
\begin{equation}\label{eq:delta}
 \delta_{n}^{-}:= 
 g^{-1}(\Delta_{n-1}^{+}) \cap  \Delta_0^{-} 
 \quand 
 \delta_{n}^{+}:= 
 g^{-1}(\Delta_{n-1}^{-}) \cap  \Delta_0^{+}.
\end{equation}
 Notice  that  
 \(
 \{ \delta_n^{-}\}_{n\geq 1}  
\) 
and 
\( 
 \{ \delta_n^{+}\}_{n\geq 1} 
\)
are $\bmod\; 0$ partitions of \( \Delta_0^-\) and \( \Delta_0^+\)  respectively and also in these cases  the partition elements depend monotonically on the index in the sense that \( n > m \) implies that \( \delta_n^{\pm}\) is closer to \( 0 \) than \( \delta_m^{\pm}\),   (and in particular the only accumulation point of these partitions is 0). Notice moreover, that 
\[
 g^{n}(\delta_{n}^{-})= \Delta_{0}^{+} \quad\text{and} \quad   g^{n}(\delta_{n}^{+})= \Delta_{0}^{-}.  \]  
We now define two non-negative integers \( n_{\pm}\) which depend on the positions of the partition elements \( \delta_{n}^{\pm}\) and on the sizes of the neighbourhoods \( U_{0\pm}\) on which the map \( g \) is explicitly defined.  
If  \( \Delta_0^{-} \subseteq U_{0-}\) and/or  \( \Delta_0^{+} \subseteq U_{0+}\), we define   \( n_{-}= 0  \) and/or \( n_{+}=0\) respectively, otherwise we let 
\begin{equation}\label{eq:n+-}
 n_{+} := \min \{n :\delta_{n}^{+} \subset U_{0+} \} \quand 
  n_{-} := \min \{n :\delta_{n}^{-} \subset U_{0-} \}.  
 \end{equation}
We can now formulate our final assumption as follows. 
  \begin{description}
    \item[\namedlabel{itm:A2}{\textbf{(A2)}}]  
There exists a \( \lambda >  1 \) such that  for all   \( 1\leq n\leq n_{\pm}\)   and for all \(    x \in \delta_n^{\pm} \)  we have \(  (g^n)'(x) > \lambda\). 
 \end{description}
Notice that \ref{itm:A2} is an expansivity condition for points outside the neighbourhoods  \( U_{0\pm}\) and \( U_{\pm 1}\) but is much weaker than assuming that the derivative of \( g\) is greater than 1 outside these neighbourhoods, which would be unnatural and unnecessarily restrictive in the presence of critical points. 
 This completes the set of conditions which we require, and for  convenience we  let  
  \[
  \widehat{\mathfrak{F}} \coloneqq \{ g : I \to I  \text{ which satisfy \ref{itm:A0}-\ref{itm:A2}}\}
  \]



The class \( \widehat{\mathfrak{F}} \) contains  many  maps which have been studied in the literature, including uniformly expanding maps and various well known intermittency maps with a single neutral fixed point. We will give a more in-depth literature review in Section~\ref{sec:lit-review}. 
Here we make a few technical remarks concerning these assumptions before proceeding to state our results in the next subsection.


\begin{rem}[Remark on notation]
  To simplify many statements which will be made through the paper, it will be useful to recall some relatively standard notation as follows.
  Given  sequences   \((s_{n})\) and \((t_{n})\) of non-negative terms, we write \( s_{n} = O(t_{n})\), or \( s_{n} \lesssim t_{n}\), if   $s_{n}/t_{n} $ is uniformly bounded  above; \( s_{n}\approx t_{n}\) if  \( s_{n}/t_{n} \) is uniformly bounded away from 0 and \( \infty\); 
  \( s_{n} = o(t_{n}) \) 
  if   $s_{n}/t_{n} \to 0$ as $n \to \infty$;  and \( s_{n}\sim t_{n} \)  if  \(s_{n}/t_{n}= 1+o(1)\), i.e.  if  \( s_{n}/t_{n}\)  converges to 1 as \( n \to \infty \). 
\end{rem}

\begin{rem}\label{rem:def}
Changing the
parameter values \( \ell_{1}, \ell_{2}, k_{1}, k_{2} \) gives rise to maps with quite different characteristics.   For example, if \( \ell_{1}>0\), we have 
\begin{equation}\label{deq_4}
g'|_{U_{-1}}(x)= 1+b_1(1+\ell_1)(1+x)^{\ell_1}
\quad\text{ and } \quad 
g''|_{U_{-1}}(x) = b_1(1+\ell_1)\ell_1 (1+x)^{\ell_1-1}. 
\end{equation}
Then  \( g'(-1)= 1 \)  and the fixed point \( -1 \) is a \emph{neutral} fixed point.   Similarly, when \( \ell_{2}>0\) the fixed point \( 1 \) is a neutral fixed point. 
On the other hand, when \( \ell_1=0\), from \eqref{eq:A1b} we have 
\begin{equation}\label{eq:ell0}
g'|_{U_{-1}}(x)= 1+b_1 + \xi'(x)
\quad\text{ and } \quad 
g''|_{U_{-1}}(x) =\xi''(x)
\end{equation}
and thus the fixed point \( -1 \) is \emph{hyperbolic repelling} with \( g'(-1)=1+b\). 
When  \( k_{1}\neq 1 \) we have 
\begin{equation}\label{deq_5}
g'|_{U_{0-}}(x)= a_1k_1|x|^{k_1-1}
\quad\text{ and } \quad 
g''|_{U_{0-}}(x)= a_1k_1(k_1-1)|x|^{k_1-2}. 
\end{equation}
  Then \( k_{1} \in (0,1) \) implies that \( |g'|_{U_{0-}}(x)|\to \infty \) as \( x \to 0 \), in which case we say that \( g|_{U_{0-}} \) has a (one-sided) \emph{singularity} at 0, whereas \( k_{1} > 1 \)   implies that \( |g'|_{U_{0-}}(x)|\to 0\) as \( x \to 0 \),  and therefore we say that \( g|_{U_{0-}} \) has a (one-sided) \emph{critical point} at 0. Analogous observations  hold for the various values of \( \ell_2\) and \( k_2\) and 
  Figure \ref{fig:fig} shows the graph of \( g \) for various  combinations of these exponents. 

For future reference we mention also some additional properties which follow from  \ref{itm:A1}. 
First of all notice that if \( \ell_1 \in (0,1)\) we have  \(g''(x) \to \infty\) but if  \( \ell_1 >1 \) we have  \(g''(x) \to 0\),  as \( x\to -1\) and, as we shall see, this qualitative difference in the higher order derivative plays a crucial role in the ergodic properties of \( g \). Analogous observations apply to \( g|_{U_{1}}\) when \( \ell_{2}>0\). Secondly, notice also that for every \( x\in U_{-1}\)  we have 
\begin{equation}\label{eq:2ndder1}
{g''(x)}/{g'(x)}  \lesssim (1+x)^{\ell_1-1} 
\end{equation}
and an analogous bound holds for \( x\in U_{1}\). Similarly,  in \(U_0\) we have 
\begin{equation}\label{deq_55}
{|g''(x)|}/{|g'(x)|} \lesssim  x^{-1},
\end{equation}
and notice that in this case the bound does not actually depend on the value of \( k_1\) or \( k_2\)  and in particular  does not depend on whether we have a critical point or a singularity.  Finally, we note that when \( \ell_1=0\), it follows from \eqref{eq:ell0} and from the assumption that  \( \xi''(x) > 0 \)  that 
\begin{equation}\label{eq:xider}
\xi'(x) > \xi(x)/(1+x)
\end{equation}
  for every \( x \in U_{-1}\). Indeed, notice that \( 1+x\) is just the distance between \( x \) and \( -1\) and thus \( \xi(x)/(1+x)\) is the slope of the straight line joining the point \( (-1, 0 ) \) to \( (x, \xi(x)) \) in the graph of \(\xi \), which is exactly the \emph{average} derivative  of \( \xi \) in the interval \([-1, x]\). Since \( \xi'' > 0 \), the derivative is monotone increasing and thus the derivative \( \xi' \) is maximal at the endpoint \( x\), which implies \eqref{eq:xider}. The same statement of course holds for \( \ell_2=0\) and for all \( x\in U_{+1} \).
\end{rem}

\subsection{Statement of Results}
\label{sec:results}

Our first  result is completely general and applies to all maps in \( \widehat{\mathfrak{F}} \).

\begin{manualtheorem}{A}\label{thm:main1}
Every \( g \in \widehat{\mathfrak{F}} \) admits a unique (up to scaling by a constant) invariant measure which is absolutely continuous with respect to Lebesgue; this measure is \( \sigma \)-finite and equivalent to Lebesgue.
\end{manualtheorem}

This is perhaps not completely unexpected but also certainly not obvious in the full generality of the maps in \( \widehat{\mathfrak{F}} \), especially for maps which admit critical points (which can, moreover, be of arbitrarily high order). Our construction gives some additional information about the measure given in Theorem \ref{thm:main1}, in particular the fact that  its density with respect to Lebesgue is locally Lipschitz and unbounded only at the endpoints \( \pm 1\). We will show that, depending on the exponents \( k_{1}, k_{2}, \ell_{1}, \ell_{2} \),  the density may or may not be integrable and so the measure may or may not be finite. More specifically, let 
\[
  \beta_1 \coloneqq k_2 \ell_1,\quad
  \beta_2 \coloneqq k_1 \ell_2,
  \quand  \beta \coloneqq \max\{ \beta_{1}, \beta_{2} \}.
\]
We will show that the density  is Lebesgue integrable at -1 or 1 respectively if and only if \( \beta_{1}\) and \( \beta_{2}\) respectively are \( <1\). In particular, letting 
\[
\mathfrak F := \{ g\in \widehat {\mathfrak F} \text{ with } \beta < 1 \}
\]
we have the following result. 
\begin{manualtheorem}{B}\label{thm:main2}
A map  \( g \in \widehat{\mathfrak F}\)  admits a unique ergodic invariant \emph{probability} measure \( \mu_g \) absolutely continuous with respect to (indeed equivalent to) Lebesgue \emph{if and only if} \( g \in \mathfrak F\).
\end{manualtheorem}

Notice that the condition \( \beta<1\) is a restriction only on the \emph{relative} values of \( k_{1}\) with respect to \( \ell_{2}\) and of \(k_{2}\) with respect to \( \ell_{1}\).  It still allows  \( k_{1}\) and/or \( k_{2} \) to be \emph{arbitrarily large}, thus allowing arbitrarily ``degenerate'' critical points, as long as the  corresponding exponents \( \ell_{2} \) and/or \( \ell_{1}\) are sufficiently small, i.e. as long as the corresponding neutral fixed points are not too degenerate.    

\medskip

We now give several non-trivial results  about the statistical properties maps  \( g \in  \mathfrak F
\) with respect to the probability measure  \( \mu_g \). To state our first result recall that 
the \emph{measure-theoretic entropy} of \( g \) with respect to the measure \( \mu \) is defined as 
\[
h_{\mu}(g):= \sup_{\mathcal P} \left\{\lim_{n\to\infty} \frac 1n \sum_{\omega_{n}\in \mathcal P_{n}} - \mu(\omega_{n})\ln \mu(\omega_{n})\right\}
\]
where the supremum is taken over all finite measurable partitions \( \mathcal P\) of the underlying measure space and \( \mathcal P_{n} := \mathcal P \vee f^{-1}\mathcal P\vee \cdots \vee f^{-n}\mathcal P \) is the dynamical refinement of \( \mathcal P\) by \( f \).

\begin{manualtheorem}{C}\label{thm:mainentropy}
  Let \( g \in \mathfrak F\). Then \( \mu_g \)  satisfies the  \emph{Pesin entropy formula:}
  \(
      h_{\mu_g}(g)=\int \log |g'|d\mu_g.
  \)
\end{manualtheorem}

For H\"older continuous  functions \( \varphi, \psi: [-1,1] \to \mathbb{R} \)  and \( n \geq 1 \), we define the \emph{correlation function}
        \[
       \mathcal C_{n}(\varphi, \psi):= \left| \int \varphi \psi\circ g^{n} d\mu - \int \varphi d\mu \int \psi d\mu \right|.
        \]
        It is well known that \( \mu_g \) is \emph{mixing} if and only if \( \mathcal C_n(\varphi, \psi) \to 0 \) as \( n \to \infty\). We say that \( \mu_g \) is \emph{exponentially mixing}, or satisfies \emph{exponential decay of correlations}  if there exists a \( \lambda > 0 \) such that  for all H\"older continuous functions \( \varphi, \psi\) there exists a constant \( C_{\varphi, \psi}\) such that \(  \mathcal C_{n}(\varphi, \psi) \leq C_{\varphi, \psi} e^{-\lambda n}\). We say that \( \mu_g \) is \emph{polynomially mixing}, or satisfies \emph{polynomial decay of correlations}, with rate \( \alpha >0\) if for all H\"older continuous functions \( \varphi, \psi\) there exists a constant \( C_{\varphi, \psi}\) such that \(  \mathcal C_{n}(\varphi, \psi) \leq C_{\varphi, \psi} n^{-\alpha}\).

\begin{manualtheorem}{D}\label{thm:maincorrelations}
  Let \( g \in \mathfrak F\). If \( \beta = 0\) then \( \mu_g\) is exponentially mixing, if \( \beta\in (0,1) \) then \( \mu_g\) is polynomially mixing with rate \( (1-\beta)/\beta\).
\end{manualtheorem} 
  
  Notice that the polynomial rate of decay of correlations \( (1-\beta)/\beta\) itself decays to 0 as \( \beta \) approaches 1, which is the \emph{transition parameter} at which the invariant measure ceases to be finite. Intuitively, as \( \beta\to 1\), the measure, while still equivalent to Lebesgue, is increasingly concentrated in neighbourhoods of the neutral fixed points, which \emph{slow down} the decay of correlations.

  \medskip
Our final result concerns a number of \emph{limit theorems} for maps  \( g \in \mathfrak F\), which depend on the parameters of the map and, in some cases, also on some additional regularity conditions.  These are arguably some of the most interesting results of the paper, and those in which the existence of two indifferent fixed points, instead of just one, really comes into play, giving rise to quite a complex scenario of possibilities. 
We start by recalling the relevant definitions. For integrable functions \( \varphi \) with \( \int \varphi d\mu = 0 \) we define the following limit theorems.

\begin{description}

  \item[\namedlabel{itm:CLT}{CLT}]
    \( \varphi\) satisfies a \emph{central limit theorem} with respect to \( \mu \) if there exists a \( \sigma^2 \geq 0 \) and a \( \mathcal{N}(0,\sigma^2) \) random variable \( V \) such that 
    \[
      \lim_{n\to \infty} \mu \left( \frac{ \sum_{ k = 0 }^{ n - 1 } \varphi \circ g^k  }{ \sqrt{n} } \leq x \right)
      = \mu ( V \leq x ),
    \]
    for every \( x \in \mathbb{R} \) for which the function \( x \mapsto \mu ( V_{\sigma^2} \leq x ) \) is continuous.

  \item[\namedlabel{itm:CLT-ns}{CLT\textsubscript{\text{ns}}}]
    \( \varphi \) satisfies a \emph{non-standard central limit theorem} with respect to \( \mu \) if there exists a \( \sigma^2 \geq 0 \) and a \( \mathcal{N}(0,\sigma^2) \) random variable \( V \) such that 
    \[
      \lim_{n\to \infty} \mu \left( \frac{ \sum_{ k = 0 }^{ n - 1 } \varphi \circ g^k  }{ \sqrt{n\log n} } \leq x \right)
      = \mu ( V \leq x ),
    \]
    for every \( x \in \mathbb{R} \) for which the function \( x \mapsto \mu ( V_{\sigma^2} \leq x ) \) is continuous.

  \item[\namedlabel{itm:SL-alpha}{SL}\textsubscript{\( \alpha \)}]
    \( \varphi \) satisfies a \emph{stable law} of index  \( \alpha \in (1,2) \), with respect to a measure \( \mu \), if  there exists a stable random variable \( W_{\alpha} \) such that
    \[
      \lim_{n\to \infty} \mu \left( \frac{ \sum_{ k = 0 }^{ n - 1 } \varphi \circ g^k }{ n^{1/\alpha} } \leq x \right)
      = \mu ( W_{\alpha} \leq x ),
    \]
    for every \( x \in \mathbb{R} \) for which the function \( x \mapsto \mu ( W_{\alpha} \leq x ) \) is continuous.
\end{description}
Finally, we say that an observable \( \varphi : [-1,1] \to \mathbb{R} \) is a \emph{co-boundary} if there exists a measurable function \( \chi : [-1,1] \to \mathbb{R} \) such that \( \varphi = \chi\circ g - \chi \).
We are now ready to state our result on the various limit theorems which hold under some conditions on the parameters and on the observable~\( \varphi \). In order to state these conditions it is convenient to introduce the following variable: 
\begin{equation}\label{eq:def-B-phi}
\beta_\varphi \coloneqq
\begin{cases}
0 &\text{if  \( \varphi ( -1 ) = 0 \) and \( \varphi ( 1 )  = 0 \) }
\\
\beta_1 &\text{if  \( \varphi ( -1 ) \neq 0 \) and \( \varphi ( 1 )  = 0 \) }
\\
\beta_2 &\text{if  \( \varphi( -1  ) = 0 \) and \( \varphi ( 1 ) \neq 0 \) } 
 \\ 
\beta &
\text{if \( \varphi( -1 )\neq 0\) and \( \varphi ( 1 )\neq 0 \)},
\end{cases}
\end{equation} 
We can then state our results in all cases in a clear and compact way as follows.

\begin{manualtheorem}{E}\label{thm:limit-theorems}
  Let \( g \in \mathfrak{F} \) and \( \varphi : [ -1, 1 ] \to \mathbb{R} \) be  H\"older continuous  with \( \int \varphi d\mu = 0 \) and satisfying 
  \begin{description}
  \item[\namedlabel{itm:H1}{(\( \mathcal H\))}]
     \(
         \nu_1 >  (\beta_{1} - 1/2)/ k_{2} \quand  \nu_{2} > ( \beta_{2} - 1/2)/ k_{1 }, 
     \)
\end{description}
where \( \nu_{1}, \nu_{2} \) are the H\"older exponents of \( \varphi|_{[-1,0]}\) and \( \varphi|_{(0,1]}\) respectively. Then 
          \begin{enumerate}
            \item
                if   \( \beta_{\varphi} \in [0,1/2 ) \) then \( \varphi \) satisfies \ref{itm:CLT},
            \item
                if   \( \beta_{\varphi} = 1/2 \)  then \( \varphi \) satisfies \ref{itm:CLT-ns},
            \item
                if   \( \beta_{\varphi} \in (1/2,1 ) \) then
                \( \varphi \) satisfies  \ref{itm:SL-alpha}\textsubscript{\( 1/ \beta_\varphi \)}.
          \end{enumerate}
           In case 3 we can replace the H\"older continuity condition \ref{itm:H1} by the weaker (in this case) condition 
             \begin{description}
  \item[\namedlabel{itm:H2}{(\( \mathcal H'\))}]
     \(
        \nu_1 >  (\beta_{1} - \beta_{\varphi})/ k_{2 } \quand  \nu_{2} >  (\beta_{2} - \beta_{\varphi} )/ k_{1 }.
     \)
\end{description}
Moreover, in all cases where~\ref{itm:CLT} holds we have that \( \sigma^2 = 0 \) if and only if \( \varphi \) is a coboundary.
\end{manualtheorem}

\begin{rem}
Our results highlight the fundamental significance of the value of the observable \( \varphi\) at the two fixed points, and how  \emph{the fixed point at which \( \varphi \) is non-zero},  in some sense \emph{dominates}, and determines the kind of limit law which the observable satisfies. If  \( \varphi \) is non-zero at both fixed points, then it is the larger exponent which dominates. 
\end{rem}

\begin{rem}
Note that  \ref{itm:H1} and \ref{itm:H2} are automatically satisfied for various ranges of \( \beta_{1}, \beta_{2} \), for example if \( \beta \leq  1/2 \) then \ref{itm:H1} always holds and  if \(  \beta = \beta_{\varphi}  \) then \ref{itm:H2} always holds.
These H\"older continuity conditions arise as technical conditions in the proof and it is not clear to us if they are really necessary and what could be proved without them. It may be the case, for example, that some limit theorems still hold under weaker regularity conditions on \( \varphi\).
\end{rem}

\begin{rem}
    We remark also that the compact statement of Theorem \ref{thm:limit-theorems} somewhat ``conceals'' quite a large number of cases which express an intricate relationship between the map parameters and the values and regularity of the observable. For example, the case \( \beta_{\varphi}=0\) allows all possible values \( \beta_1, \beta_2\in [0,1)\) and the case \( \beta_{\varphi}=\beta_1\) allows all possible values of \( \beta_2\in [0,1)\). We therefore have a huge number of possible combinations which do not occur in the case of maps with just a single intermittent fixed point. 
\end{rem}

\subsection{Examples and Literature Review}
\label{sec:lit-review}

There is an extensive literature  on the dynamics and statistical properties of full branch maps, which have been studied systematically since the 1950s. Their importance stems  partly from the fact that they occur very naturally, for example any smooth non-invertible local diffeomorphism of \( \mathbb S^1 \) is a full branch map, but also, and perhaps most importantly, because many arguments in Smooth Ergodic Theory apply in this setting in a particularly clear and conceptually straightforward way.  Indeed, arguably,  most existing techniques used to study hyperbolic (including non-uniformly hyperbolic) dynamical systems are essentially (albeit often highly non-trivial) extensions and generalisations of methods first introduced and developed in the setting of one-dimensional full branch maps. 

Our class of maps \( \widehat{\mathfrak F}\) is quite general and includes many one-dimensional full branch maps which have been studied in the literature as well as many maps which have \emph{not} been previously  studied. We give below a brief survey of some of these examples and indicate for which choices of parameters these correspond to maps in our family\footnote{Recall that we have fixed the  domains of the branches of our maps as \( [-1,0)\) and \( (0, 1]\)  for convenience. In the examples below, when listing parameters, we slightly abuse notation and assume an affine change of coordinates which transforms the given domains into the ones used in our class.}.

Arguably one of the very first and simplest general class of maps for which the existence of an invariant ergodic and absolutely continuous probability measure was proved are \emph{uniformly expanding full branch} maps with derivatives uniformly bounded away from 0 and infinity, a result often referred to as the \emph{Folklore Theorem} and generally attributed to Renyi. 
Some particularly simple examples of uniformly expanding maps are piecewise affine maps such as those given by 
\begin{equation}\label{eq:pamap}
      g (x) = 
      \begin{cases} 
      ax &\text{for } x \in [0,1/a] \\ 
      \frac{a}{a-1} \left( x - \frac{1}{a} \right)  &\text{for } x \in (1/a, 1]
      \end{cases} 
\end{equation}
for  parameters \( a > 1 \), see Figure \ref{fig2a}. These are  easily seen to be  contained in the class \( \widehat{\mathfrak{F}} \)  with parameters \( (\ell_1,\ell_2,k_1,k_2,a_1,a_2,b_1,b_2) = ( 0, 0, 1, 1, a, {a}/({a-1}), a - 1, {a}/({a-1}) - 1) \).

  In the late '70s, physicists Manneville and Pomeau \parencite{PomMan80} introduced a simple but extremely interesting generalisation consisting of a class of full branch one-dimensional maps \( g: [0, 1] \to [0,1] \),
 which they called \emph{intermittency maps}, defined by 
 \begin{equation}\label{eq:pm}
    g (x) = x( 1 + x^{\alpha} )\mod 1 
 \end{equation}
   for  \( \alpha > 0 \), see Figure \ref{fig2b} (notice that for \( \alpha = 0 \) this just gives the map \( g(x) = 2x \) mod 1, which is just  \eqref{eq:pamap} with \( a =2\)).  These maps can be seen to be contained in our class \( \widehat{\mathfrak{F}} \) by taking the parameters \( (\ell_1,\ell_2,k_1,k_2,a_1,a_2,b_1,b_2) = ( \alpha, 0, 1, 1, a, a, 1, 1) \), where \( a = g'(x_0) \), and \( x_0 \in (0,1) \) is the boundary of the intervals on which the two branches of the map are defined. 
  The Manneville-Pomeau maps are interesting because  the uniform expansivity condition fails at a single fixed point on the boundary of the interval, where we have \( g'(0)=1\)  Their motivation was to model fluid flow where long period of stable flow is followed with an intermittent phase of turbulence, and they showed that this simple model indeed seemed to exhibit such dynamical behaviour.  It was then shown in \parencite{Pia80} that for \( \alpha>2\),  the intermittency maps \emph{failed} to have an invariant ergodic and absolutely continuous probability measure and satisfies the extremely remarkable property that \emph{the time averages of Lebesgue almost every point converge to the Dirac-delta measure \( \delta_0\) at the neutral fixed point}, even though these orbits are dense in \( [0,1]\) and the fixed point is topologically repelling.

   \begin{figure}[h] 
\centering
\begin{subfigure}{0.3\textwidth}
\includegraphics[width=\linewidth]{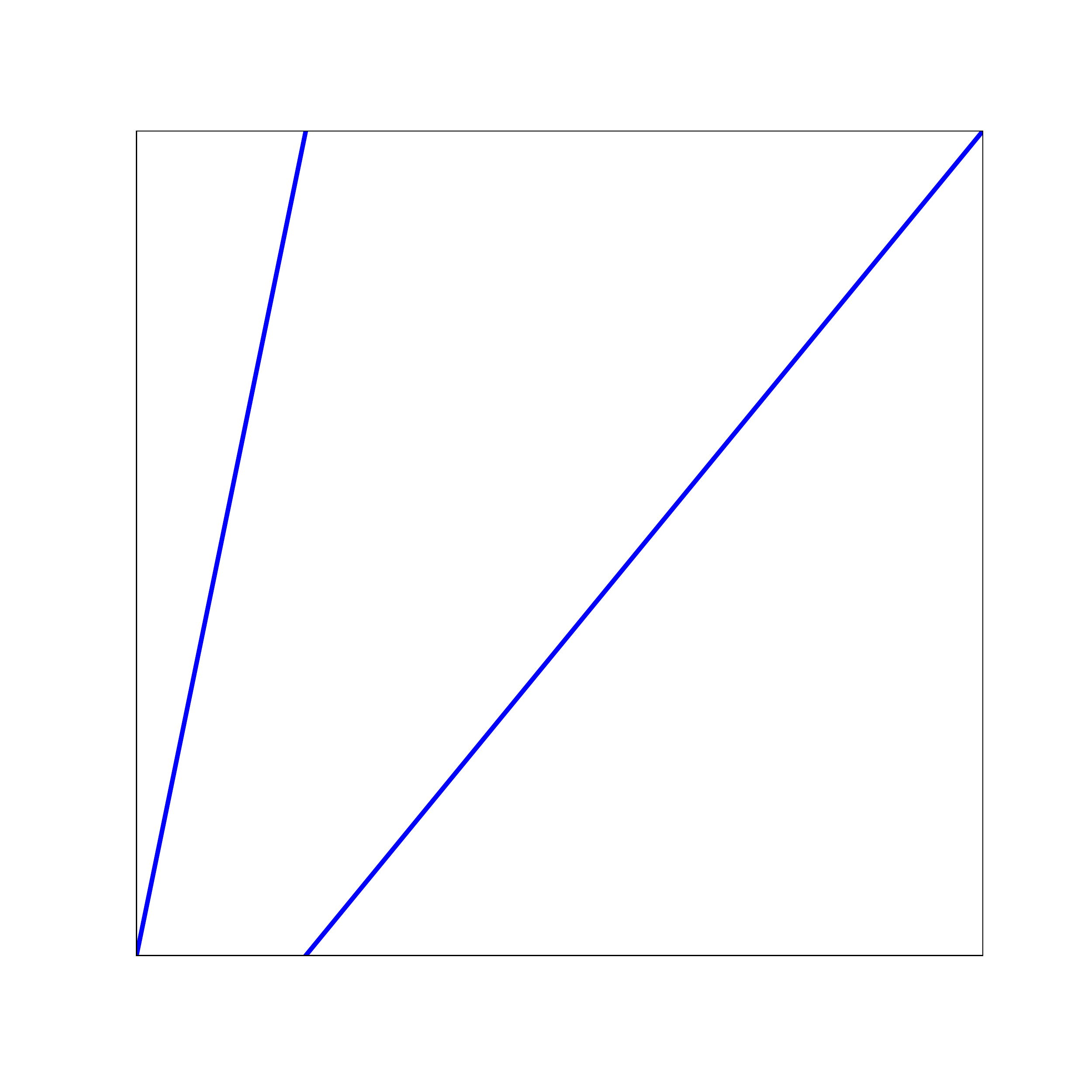}
\caption{Graph of \eqref{eq:pamap} with \( a = 5 \)}
\label{fig2a}
\end{subfigure}
\begin{subfigure}{0.3\textwidth}
\includegraphics[width=\linewidth]{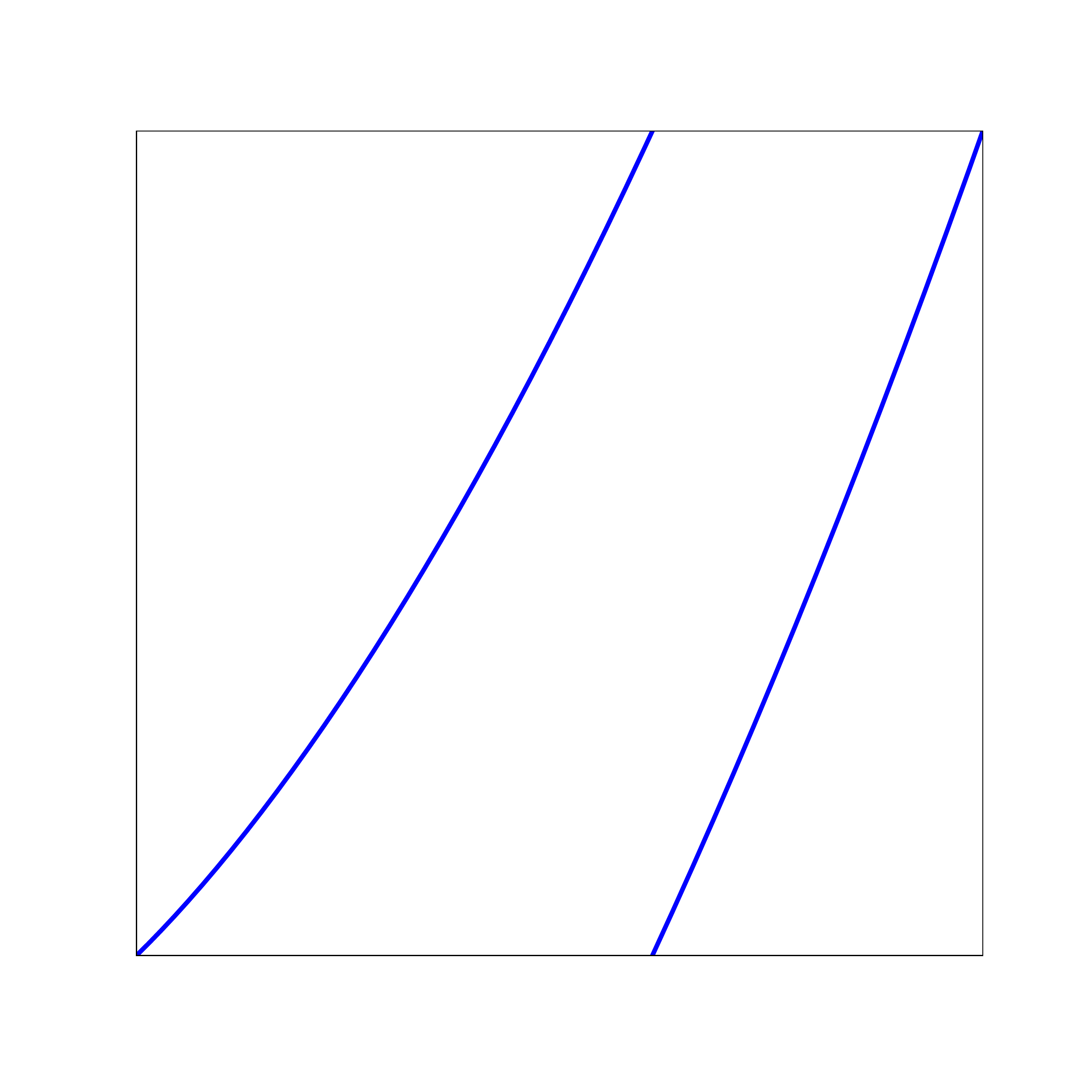}
\caption{Graph of \eqref{eq:pm} with \( \alpha = 9/10 \)}
\label{fig2b}
\end{subfigure}
\caption{Graphs of a piecewise affine, Manneville-Pomeau, and Liverani-Saussol-Vaienti maps.}
\end{figure}

Various variations of intermittency maps have been studied extensively from various points of views and with different techniques yielding quite deep results,  see e.g.  \cite{LivSauVai99,You99, Sar01, Mel09, PolSha09, FisLop01, Gou04,Gou04a, NicTorVai18, CoaHolTer21, FreFreTodVai16, Kor16,BahSau16,Ter16a,Shevan13, Zwe03}.  
One well known version is the so-called Liverani-Saussol-Vaienti (LSV) map \( g: [0, 1] \to [0,1] \) introduced in \cite{LivSauVai99} and defined by 
 \begin{equation}\label{eq:lsv}
  g (x) = \begin{cases} x (1 + 2^\alpha x^{\alpha}) & \text{for } x\in [0,1/2] \\ 2x - 1 &\text{for } x \in (1/2,1) \end{cases} 
  \end{equation}
  with parameter \( \alpha > 0 \), see Figure \ref{fig:lsv}. This maintains the essential features of the Maneville-Pomeau maps \eqref{eq:pm}, i.e. it is uniformly expanding except at the neutral fixed point at the origin, but in slightly simplified form where the two branches are always defined on the fixed domains \( [0,1/2]\) and \( 1/2, 1)\) and the second branch is affine, both of which make the map family easier to study, including the effect of varying the parameter.   The family of LSV maps \eqref{eq:lsv} can be seen to be contained in our class \(  \widehat{\mathfrak{F}}  \) by taking the parameters \( (\ell_1,\ell_2,k_1,k_2,a_1,a_2,b_1,b_2) = ( \alpha, 0, 1, 1, 2, 2, 2^{\alpha}, 1) \). 

In an earlier paper  \cite{Pik91}, Pikovsky  had introduced the maps \( g : \mathbb{S}^1 \to \mathbb{S}^1 \), defined (in a somewhat unwieldy way)  by the implicit equation
  \begin{equation}\label{eq:pik}
    x = 
    \begin{cases} 
      \frac{1}{2 \alpha} (1 + g(x))^{\alpha} & \text{for } x\in [0, 1/2\alpha] \\
      g(x) + \frac{1}{2 \alpha } ( 1 - g(x) )^{\alpha} &\text{for } x \in (1/2\alpha,1)
    \end{cases} 
  \end{equation}
for \( x\in [0,1)\), and then by the symmetry \( g(x) = g(-x) \) for \( x\in  (-1, 0]\), see Figure \ref{fig:vaienti}. These maps have a neutral fixed point at the left end point, like in \eqref{eq:pm} and \eqref{eq:lsv}  but with the added complication of having unbounded derivative at the boundary between the domains of the two branches. On the other hand the definition is specifically designed in such a way that the order of intermittency is the inverse of the order of the singularity and, together with the  symmetry of the two branches, this implies that Lebesgue measure is invariant for all values of the parameter \( \alpha > 0 \). Ergodic and statistical properties of these maps were studied in \cite{AlvAra04,CriHayMarVai10,BosMur14} and they can be seen to be contained  within our class \( \widehat{\mathfrak{F}} \) by taking the parameters \( (\ell_1,\ell_2,k_1,k_2,a_1,a_2,b_1,b_2) = ( \alpha - 1, \alpha - 1, 1/\alpha, 1/\alpha, (2\alpha)^{1/\alpha}, (2\alpha)^{1/\alpha}, 1/2\alpha, 1/2\alpha) \).

Finally,  \cite{Ino92,Cui21} consider a class of maps, see Figure \ref{fig:cui} for an example, with a single intermittent fixed point and multiple critical points with each critical point mapping to the fixed point. 
These include some maps which are more general than those we consider here as they are defined near the fixed and critical points  through some bounds rather than explicitly as we do here, but are also more restrictive as they only allow for a single neutral fixed point. Under a condition on the product of the orders of the neutral and (the most degenerate) critical point which is exactly analogous to our condition \( \beta<1 \), the existence of an invariant ergodic probability measure is proved which exhibits decay of correlations but no bounds are given for the rate of decay and no limit theorems are obtained.

\begin{figure}[h]
\centering
\begin{subfigure}{0.3\textwidth}
    \centering
    \includegraphics[width=\linewidth]{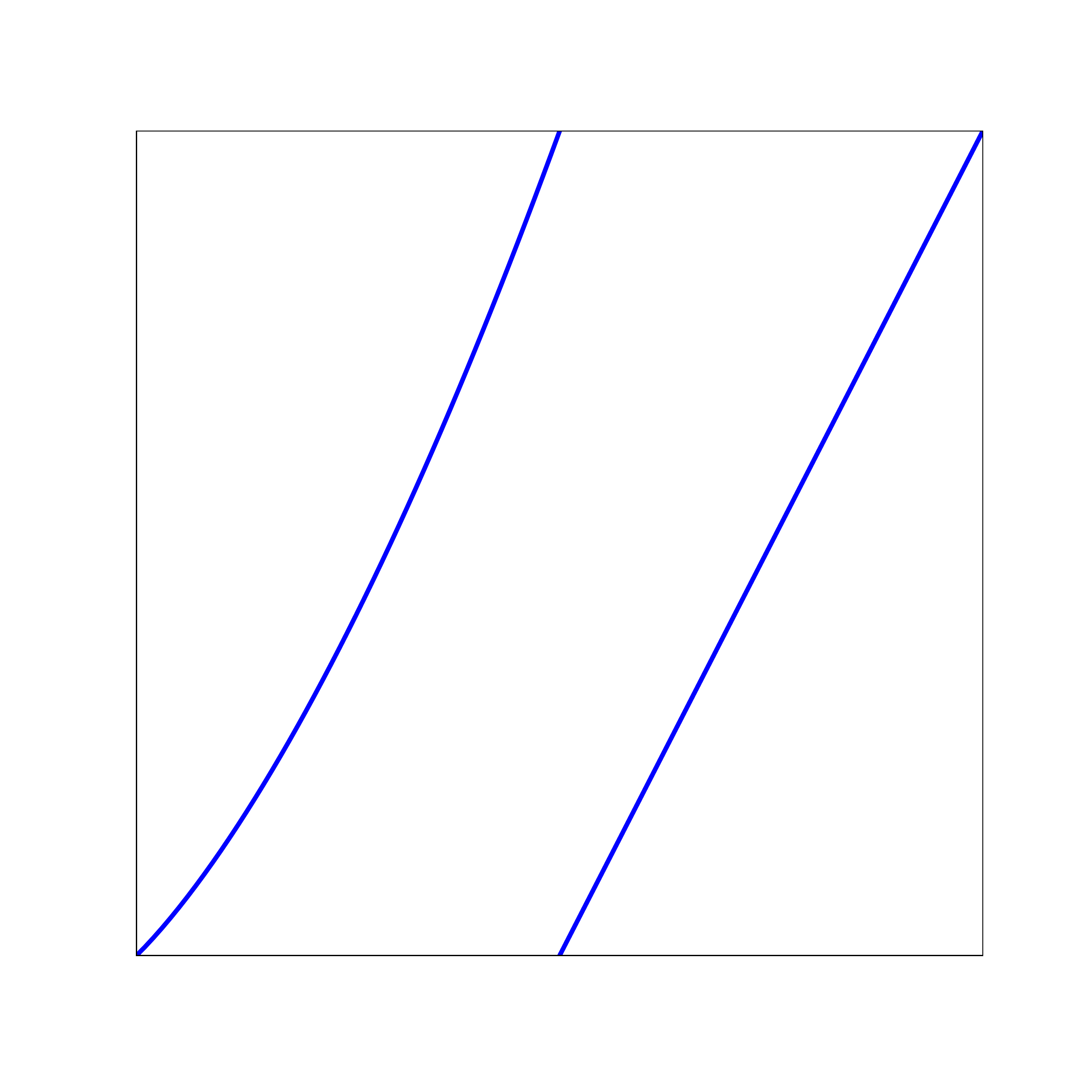}
    \caption{Graph of \eqref{eq:lsv} with \( \alpha = 9/10 \)}
    \label{fig:lsv}
\end{subfigure}
\begin{subfigure}{0.3\textwidth}
    \centering
    \includegraphics[width=\linewidth]{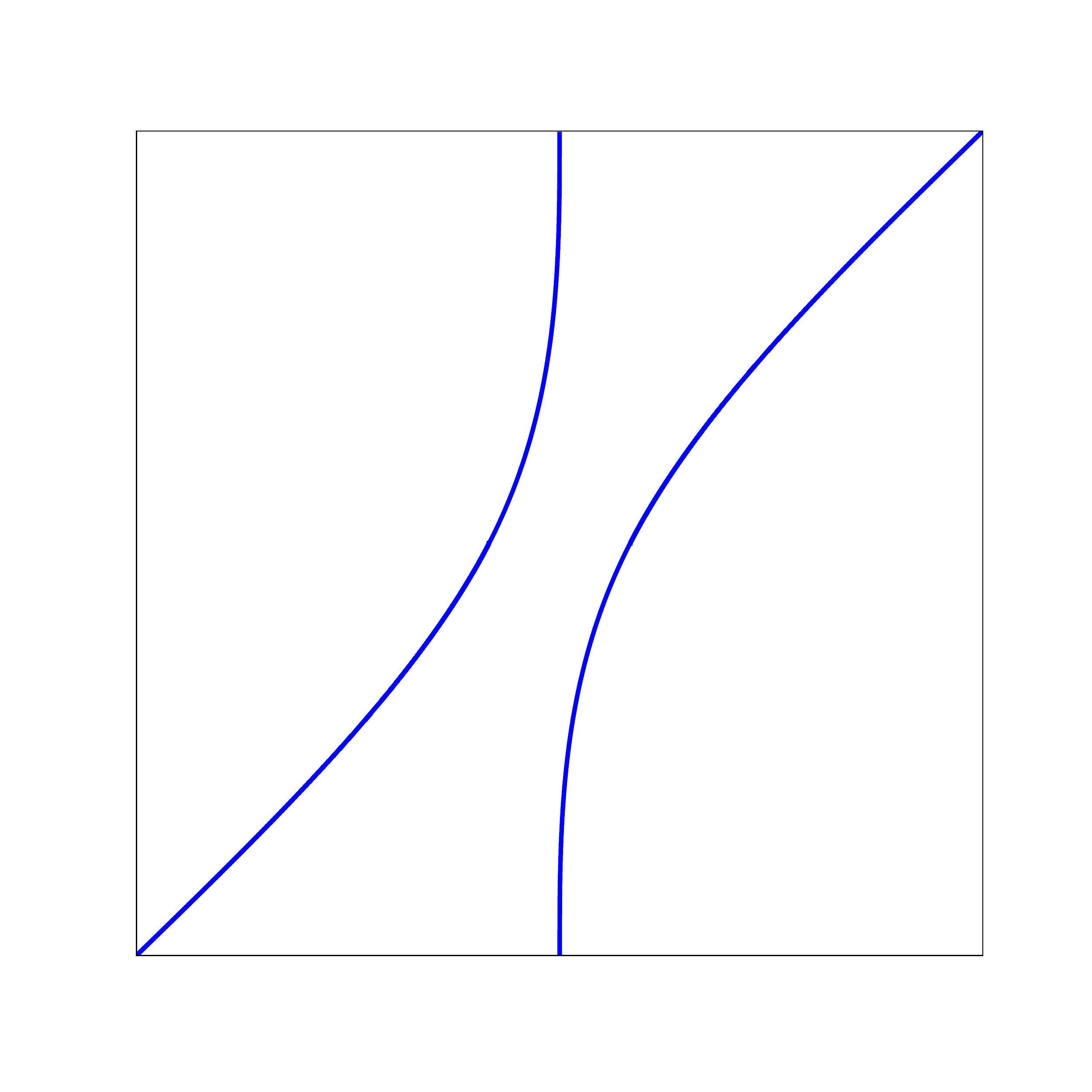}
    \caption{Graph of \eqref{eq:pik} with \( \alpha = 3 \)}
    \label{fig:vaienti}
\end{subfigure}
\begin{subfigure}{0.3\textwidth}
    \centering
    \includegraphics[width=\linewidth]{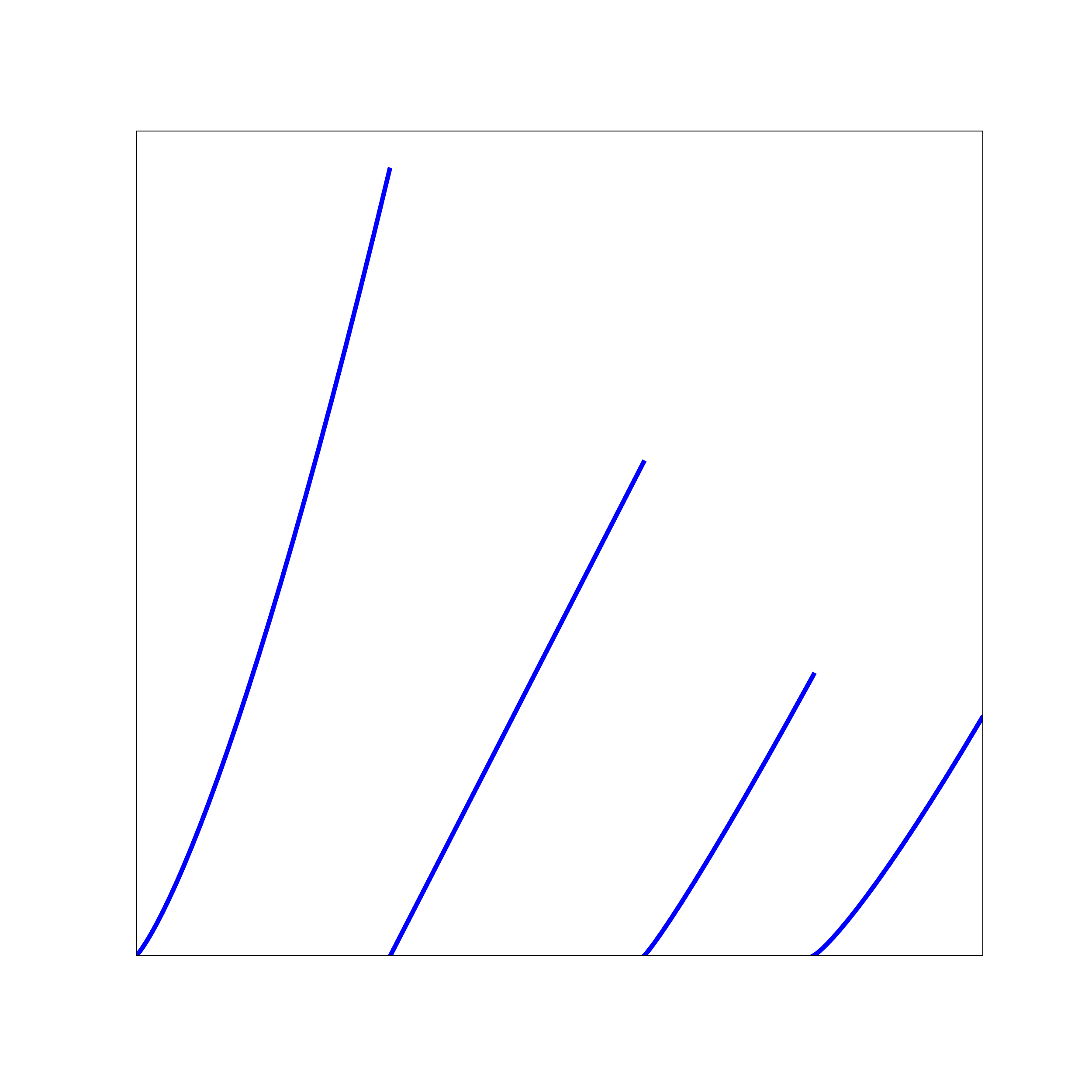}
    \caption{An example from \cite{Cui21}}
    \label{fig:cui}
\end{subfigure}
\label{fig:generalisations-of-pm}
\caption{Graphs of previsously studied generalisations of the Manneville-Pomeau map.}
\end{figure}

\begin{figure}[h]
\begin{minipage}{.4\textwidth}
\end{minipage}
\begin{minipage}{.6\textwidth}

\end{minipage}
\end{figure}

 \section{Overview of the proof}
\label{sec:outline}
We discuss here our overall strategy and prove our Theorems modulo some key technical Propositions which we then prove in the rest of the paper. 
Our argument can be naturally divided into three main steps which we describe in some detail in the following three subsections. 

\subsection{The induced map}
The first step of our arguments is the 
construction of an \emph{induced} full branch Gibbs-Markov map, also known as a \emph{Young Tower}. This is relatively standard for many  systems, including intermittent maps, however, the inducing domain which we are obliged to use here due to the presence of two indifferent fixed points is \emph{different from the usual inducing domains} and requires a more sophisticated \emph{double inducing} procedure, which we outline here and describe and carry out in detail in Section~\ref{sec:induced}.  Recall the definition of \( \Delta_0^- \) in~\eqref{eq:Delta0} and, for \( x\in \Delta_0^-\), let 
\[
\tau (x) \coloneqq \min \{ n > 0 : g^{n} (x) \in \Delta_0^{-}\} 
\] 
be the first return time to \(\Delta_0^-\).  Then we define the \emph{first-return induced map} 
\begin{equation}\label{eq:G}
G: \Delta_0^- \to \Delta_0^- \quad \text{ by } \quad G(x) \coloneqq g^{\tau(x)} (x).
\end{equation}
We say that a first return map (or, more generally, any induced map), \emph{saturates} the interval \( I \) if 
  \begin{equation}\label{eq:sat}
 \bigcup_{n\geq 0} \bigcup_{i= 0}^{n-1} g^{i}(\{\tau =  n \})
= \bigcup_{n\geq 0} g^{n}(\{\tau >   n \})
   = I \ (\text{mod} \ 0). 
  \end{equation}
  Intuitively, saturation means that the return map ``reaches'' every part of the original domain of the map \( g \), and thus the properties and characteristics of the return map reflect, to some extent, all the relevant characteristics of \( g \). 
  
  \begin{rem}\label{rem:disjoint}
   If \( G \) is a first return induced map, as in our case, then all sets of the form \( g^{i}(\{\tau =  n \}) \) are pairwise disjoint and therefore form a  partition of \( I \) mod 0.
  \end{rem}

The first main result of the paper is the following. 

\begin{prop}\label{thm:inducedmap}  Let \( g \in \widehat{\mathfrak{F}} \). Then  
  \( G: \Delta_0^- \to \Delta_0^-  \) is a first return induced  Gibbs-Markov map which saturates \( I \).
\end{prop}

We give the precise definition of  Gibbs-Markov map, and prove Proposition \ref{thm:inducedmap}, in Section~\ref{sec:induced}. In 
Section \ref{sec:top} we describe the topological structure of \( G \) and show that it a full branch map with countably many branches which saturates \( I \) (we will define \( G \) as a composition of two full branch maps, see \eqref{def:Gtilde} and \eqref{def:G-}, which is why we call the construction  a double inducing procedure); in Section~\ref{sec:est} we obtain key estimates concerning the sizes  of the partition elements of the corresponding partition; in Section \ref{sec:exp} we show that \( G \) is uniformly expanding; in Section~\ref{sec:dist} we show that \(G \)  has bounded distortion. From these results we get Proposition~\ref{thm:inducedmap} from which we can then obtain our first main Theorem.

\begin{proof}[Proof of Theorem \ref{thm:main1}]
By  standard results \( G \) admits a unique ergodic invariant probability measure \( \hat{\mu}_- \), supported on \( \Delta_0^- \),  which is equivalent to Lebesgue measure~\( m \) and which has   Lipschitz continuous  density 
\( 
    \hat h_-= d\hat\mu_-/dm
\)   bounded above and below. 
We then  ``spread'' the measure over the original interval \( I \) by defining the  measure 
\begin{equation}
  \label{eq:mu}
  \tilde\mu  \coloneqq \sum_{n=0}^{\infty} 
  g^n_*(\hat\mu_-|\{\tau \geq  n\})
\end{equation}
where \( 
 g^n_*(\hat\mu_-|\{\tau \geq  n\})(E):=   \hat{\mu}_- ( g^{-n} ( E ) \cap \{ \tau \geq  n \} ). 
\) Again by standard arguments,  we have that \( \tilde\mu \) is a sigma-finite measure which is ergodic and invariant for \( g \) and, using the non-singularity of \( g \), it is absolutely continuous with respect to Lebesgue. The fact that  \( G\) saturates \( I \) implies moreover that  \( \tilde \mu \) is equivalent to Lebesgue, which completes the proof. 
\end{proof}

\begin{rem}\label{rem:inducedmap}
We emphasize that we are not assuming \emph{any} symmetry  in the  two branches  of the map~\( g \). It is not important that the branches are defined on intervals of the same length and, depending on the choice of constants, we might even have a critical point in one branch and a singularity with unbounded derivative on the other. Interestingly, however, there is some symmetry in the construction in the sense that  for  \( x\in \Delta_0^+\), we can define the first return map 
\(
G_+: \Delta_0^+ \to \Delta_0^+ \) in a completely analogous way to the definition of \( G\) above (see discussion in Section \ref{sec:top}). Moreover, the conclusions of Proposition \ref{thm:inducedmap} hold for \( G_+\) and thus \( G_+\) admits a unique ergodic invariant probability measure \( \hat{\mu}_+ \) which is equivalent to Lebesgue measure~\( m \) and such that the density \(\hat h_+:= d\hat\mu_+/dm\)
is Lipschitz continuous and bounded above and below. 
The two maps \( G\) and \( G_+\) are clearly distinct, as are the measures \( \hat \mu_-\) and \( \hat \mu_+ \), but exhibit a  subtle kind of symmetry in the sense that the corresponding  measure \( \tilde \mu\) obtained by substituting \( \hat \mu_- \) by \( \hat \mu_+ \) in \eqref{eq:mu} is,  up to  a constant scaling factor, exactly the same measure.
\end{rem}

\begin{cor}\label{cor:density}
The density \( \tilde h \) of \( \tilde \mu|_{\Delta_0^- \cup \Delta_0^+} \) is Lipschitz continuous and bounded and \( \tilde \mu |_{\Delta_0^-} = \hat \mu\). 
\end{cor}

\begin{proof}
Since \( G \) is a first return induced map it follows that the measure \( \tilde \mu \) defined in \eqref{eq:mu} satisfies \( \tilde \mu |_{\Delta_0^-} = \hat \mu\) and so the density \( \tilde h \) of \( \tilde \mu \) is Lipchitz continuous and bounded away from both \( 0 \) and infinity on \( \Delta_0^- \).
Moreover, as mentioned in Remark \ref{rem:inducedmap}, \( \tilde \mu | _{\Delta_0^+}\) is equal, up to a constant, to the measure \( \hat \mu_{+} \) and so the density of   \( \tilde \mu |_{\Delta_0^+} \) is also Lipschitz continuous and bounded away from 0 and infinity.
\end{proof}

\begin{rem}\label{rem:mu}
We have used above the notation \( G \) rather than \( G_{-}\) for simplicity as this is the map which plays a more central role in our construction, see Remark \ref{rem:G} below. Similarly, we will from now on simply use the notation \( \hat\mu\) to denote the measure 
\( \hat \mu_- \). 
\end{rem}

\subsection{Orbit distribution estimates}

The second step of the argument is aimed at establishing conditions under which the measure~\( \tilde \mu \) is finite, and can therefore be renormalized to a probability measure
\(
\mu:= \tilde \mu/ \tilde \mu(I), 
\)
and aimed at studying the ergodic and statistical properties of \( \mu \). Our approach here differs even more significantly from existing approaches in the literature, although it does have some similarities with the argument of \cite{CriHayMarVai10}: rather than starting with estimates of the tail of the inducing time (which would themselves anyway be significantly more involved than in the usual examples of intermittency maps with a single critical point due to our double inducing procedure), we carry out  \emph{more general estimates} on the \emph{distribution} of iterates of points in \( I_-\) and \( I_+\) before they return to \( \Delta_{0}^{-}\). More precisely,  we define the functions \( \tau^\pm(x): \Delta_0^-\to \mathbb N \)  by  

\begin{equation}\label{eq:taupm}
\tau^+(x):= \#\{1\leq i \leq \tau: g^i(x)\in I_+ \},  
\quand
\tau^-(x):= \#\{1\leq j \leq \tau: g^j(x)\in I_- \}.
\end{equation}
These functions  \emph{count  the number of iterates of \( x \) in \( I_-\) and \( I_+\) respectively before returning to~\( \Delta_0^-\)}.  
Then for any \( a,b\in \mathbb R \) we define \emph{weighted combination}  \(\tau_{a,b} : \Delta_-^0 \to \mathbb R \) by 
\begin{equation}\label{eq:def-of-tau}
  \tau_{a,b} (x) = a\tau^+(x) + b \tau^-(x) 
\end{equation}
As we shall see as part of our construction of the induced map, both of these functions are \emph{unbounded} and their level sets have a non-trivial  structure in \( \Delta_-^0 \) and, moreover, the \emph{inducing time} function \( \tau: \Delta^{-}_{0}\to \mathbb N \) of the induced map \( G_{-}\) corresponds exactly to \( \tau_{1,1}\) so that 
\begin{equation}\label{eq:tautau}
\tau(x) = \tau_{1,1}(x) = \tau^+(x) + \tau^-(x).  
\end{equation}

The key results of this part of the proof consists of explicit and  sharp asymptotic bounds for the distribution of \( \tau_{a,b} \) for different values of \( a,b\), from which we can then obtain as an immediate corollary the rates of decay of the inducing time function \( \tau\),  and which will also provide the core estimates for  the various distributional limit theorems.
To state our results, let
\begin{equation}\label{eq:B}
    B_1 := a_1^{-1/k_1} (\ell_2 b_2)^{-1/\beta_{2}}
    \quand
     B_2 := a_2^{-1/k_2} (\ell_1 b_1)^{-1/\beta_{1}},
\end{equation}
(the expressions defining the constants \( B_{1}, B_{2}\) will appear in the proof of Proposition \ref{prop:y_n-and-delta_n} below).

Recall from Corollary \ref{cor:density} that the density \( \tilde h \) of \( \tilde\mu\) is  bounded on \( {\Delta_0^- \cup \Delta_0^+}  \) and  let \( \tilde h(0^-)\) and \( \tilde h(0^+)\) denote the values  of this density on either side of \( 0 \) . Then,   for any  \( a,b \geq 0 \), we let 
\begin{equation}
  \label{eq:def-of-Ca-Cb}
  C_a \coloneqq  \tilde h(0^-)   B_1 a^{{1}/{\beta_2}}, \quand
  C_b \coloneqq  \tilde h(0^+)   B_2 b^{{1}/{\beta_1}}.
\end{equation}
Then we have the following distributional estimates.

\begin{prop}\label{prop:tail-of-tau}
  Let \( g \in \widehat{\mathfrak{F}} \). Then for every  \( a, b \geq 0 \) we have the following distribution estimates. 

  For every \( \gamma \in [0,1) \)
      \begin{align}
      \label{eq:a+b-tail}
       \tilde \mu( a\tau^+ + b\tau^- > t ) &= 
      \begin{cases}
        C_b t^{-1/\beta_1} + C_a t^{-1/\beta_2} 
     + o ( t^{ - \gamma  - 1 / \beta })
        &\text{if } \ell_1, \ell_2 > 0 \\
        C_b t^{-1/\beta_1} 
  + o( t^{ - \gamma -1/\beta_1} ) 
        &\text{if } \ell_1 >  0, \ell_2 = 0 \\
        C_a t^{-1/\beta_2} 
        + o( t^{-\gamma -1/\beta_2} ) 
        &\text{if } \ell_1 =  0, \ell_2 > 0 \\
        O \left( (1 + b_1)^{-t/k_2} + (1 + b_2)^{-t/k_1} \right)
        &\text{if } \ell_1 = 0 , \ell_2 = 0
      \end{cases} \\[5pt]
      \label{eq:a-b-pstve-tail}
      \tilde \mu ( a\tau^+ - b\tau^-  > t ) &= 
      \begin{cases}
        C_a t^{-1/\beta_2} 
        + o( t^{- \gamma - 1/\beta} ) 
        \hspace{2cm}&\text{if } \ell_2  > 0 \\
        O\left((1 + b_2)^{-t/k_1}\right),
        &\text{if } \ell_2 =  0
      \end{cases} \\[5pt]
      \label{eq:a-b-neg-tail}
      \tilde \mu ( a\tau^+ - b\tau^- < - t ) &=
      \begin{cases}
        C_b t^{-1/\beta_1} 
        + o( t^{-\gamma - 1/\beta} ) 
         \hspace{2.1cm} &\text{if } \ell_1  > 0 \\
       O\left((1 + b_1)^{-t/k_2}\right),
        &\text{if } \ell_1 =  0.
      \end{cases}
    \end{align}
\end{prop}

\begin{rem}\label{rem:dist-tau-a-b}
We have assumed in Proposition \ref{prop:tail-of-tau} that \( a,b \geq 0 \) to avoid stating explicitly too many cases, but one can easily read off the tails for \( \tau_{a,b} \) for arbitrary \( a,b \in \mathbb{R} \).  For example, if \( a<0\) and \( b>0\) we can write  \( \tilde \mu ( -a\tau^+ + b \tau^- > t ) = \tilde\mu ( a\tau^+ - b \tau^- < -t ) \) and get the corresponding estimate from \eqref{eq:a-b-neg-tail}. Notice moreover, that the estimates for \( \ell_{1}=0\) and/or \( \ell_{2}=0\) are exponential.
\end{rem}


Recall from Corollary \ref{cor:density} that \( \hat\mu = \tilde\mu \) on the inducing domain \( \Delta_{0^{-}}\) and therefore all the above estimates hold for \( \hat\mu \) with exactly the same constants. In particular by  Proposition \ref{prop:tail-of-tau} and   \eqref{eq:tautau}, we immediately get the corresponding estimates for  the tail  \( \hat \mu ( \tau > t ) = \tilde \mu ( \tau > t ) \). 

\begin{cor}
  \label{cor:tail-of-return-time}
 If \( \beta = 0 \) then \( \hat\mu ( \tau > t ) \) decay exponentially as \( t \to +\infty \). 
  If \( \beta > 0 \) then there exists a positive constant \( C_{\tau} \) (which can be computed explicitly) such that
  \[
     \hat \mu ( \tau > t ) \sim C_{\tau} t^{-1/\beta}.
  \]
\end{cor}
Proposition \ref{prop:tail-of-tau} will be proved in Section \ref{sec:tail}, here we show how it implies Theorems \ref{thm:main2}, \ref{thm:mainentropy}, \ref{thm:maincorrelations}.

\begin{proof}[Proof of Theorems \ref{thm:main2}, \ref{thm:mainentropy}, and \ref{thm:maincorrelations}]
From the definition of  \( \tilde \mu\) in  \eqref{eq:mu} and since  \( g^{-n}(I)=I \) we have 
\[
   \tilde\mu(I)     \coloneqq \sum_{n=0}^{\infty} \hat\mu_-(g^{-n}(I)\cap\{\tau > n\})
   =\sum_{n=0}^{\infty} \hat\mu_-(I\cap\{\tau > n\})
                    =\sum_{n=0}^{\infty} \hat\mu_-(\tau > n).
   \]
   By Corollary \ref{cor:tail-of-return-time}, if \( \beta=0\), the quantities \(  \hat\mu_-(\tau > n) \) decay exponentially and, if \( \beta >0 \) we have 
   \[
   \tilde\mu(I)  
                    = C\sum_{n=1}^{\infty}n^{-\frac{1}{\beta} } ( 1 + o(1) ),
   \]
   for some \( C > 0 \).
This implies that \(    \tilde\mu(I)  < \infty \) if and only if \( \beta \in [0,1) \), i.e. if and only if \( g \in \mathfrak F\). Thus, for  \( g \in \mathfrak F\) we can define the measure \( \mu_{g}:=\tilde \mu/\tilde \mu(I)\), which is an invariant ergodic probability measure for \( g \),  and  is unique because it is equivalent to Lebesgue, thus proving Theorem  \ref{thm:main2}. 
Theorem \ref{thm:mainentropy} follows from Theorem A in \cite{AlvMes21} by noticing that  \(\mathcal{P}=\{(-1,0),(0,1)\}\) is a Lebesgue mod 0 generating partition  such that  \(H_{\mu_g}(\mathcal{P})<\infty\) and \(h_{\mu_g}(g,\mathcal{P})<\infty\), and therefore  \(h_{\mu_g}(g)<\infty\). 
Finally, Theorem~\ref{thm:maincorrelations} follows by well known results  \cite{You99} which show that the decay rate of the tail of the inducing times provides upper bounds for the rates of decay of correlations as stated. 
\end{proof}

  
\subsection{Distribution of induced observables}
\label{sec:distind}

The last part of our argument is focused on obtaining the limit theorems stated in Theorem \ref{thm:limit-theorems}.  When  \( \beta=0 \) the decay of correlations is exponential and the result follows from \cite{You99}. Similarly, after having established Proposition \ref{thm:inducedmap} and Corollary \ref{cor:tail-of-return-time}, the case that only one of \( \ell_{1} , \ell_{2} \) is positive implies that there is only one  intermittent fixed point, and thus essentially reduces to the argument given in \cite[Theorem 1.3]{Gou04} for the LSV map. We only therefore  need to consider the case that both \( \ell_{1}, \ell_{2} > 0 \), which implies in particular that \( \beta\in (0,1)\). 

Given an observable $ \varphi : [0,1] \to \mathbb{R}$, we define  the  \emph{induced observable} 
\(
 \Phi : \Delta_0^- \to \mathbb{R} 
 \) 
by 
 \[
  \Phi(x) \coloneqq \sum_{ k = 0 }^{ \tau (x) - 1 } \varphi \circ g^k. 
  \] 
\begin{defn}\label{def:def-of-D-alpha}
We write   \( \Phi \in \mathcal{D}_{\alpha} \) if
$
    \exists c_{1}, c_{2} \geq 0 $,  with \emph{at least one} of \( c_1,c_2 \) non-zero, such that 
 \begin{equation}\label{eq:def-D-alpha}
     \hat{ \mu } ( \Phi > t ) = c_{1}t^{-\alpha} + o( t^{-\alpha} ) 
     \quand
    \hat{ \mu } ( \Phi < -t ) = c_{2}t^{-\alpha} + o( t^{-\alpha} ), 
 \end{equation}
\end{defn}

In certain settings, limit theorems can be deduced from properties of the induced observable~\( \Phi \). In particular,  it is proved in Theorems 1.1 and 1.2 of   \cite{Gou04} that, precisely in our setting \footnote{The assumptions of \cite[Theorems 1.1 and 1.2]{Gou04}
are that \( \varphi \)  is H\"older continuous and \( G \) is an 
 induced Gibbs-Markov map    with invariant absolutely continuous probability measure \( \hat \mu \) and return time satisfying  \( \hat\mu ( \tau > n ) = O ( n^{-\gamma} ) \) for some \( \gamma>1\), which holds in our case by Corollary \ref{cor:tail-of-return-time} and the fact that \( \beta \in (0,1)\).}: 
\begin{align}
  \label{eq:cond-clt}
  & \text{ if } \Phi \in L^2 ( \hat{\mu})   \text{ then \( \varphi \) satisfies } \ref{itm:CLT}, \\
  \label{eq:cond-clt-ns}
  & \text{ if } \Phi \in \mathcal{D}_{2} \text{ then \( \varphi \) satisfies }\ref{itm:CLT-ns}, \\
  \label{eq:cond-SL}
  & \text{ if } \Phi \in \mathcal{D}_{\alpha} \text{ with } \alpha \in (1, 2) \text{ then \( \varphi \) satisfies } \ref{itm:SL-alpha}_{\alpha}.
\end{align}
We will argue that in each case of Theorem \ref{thm:limit-theorems}, the induced observable \( \Phi \) satisfies one of the above. 
To prove this, we first  decompose a general  observable \( \varphi: [-1, 1]\to \mathbb R\) by 
letting  \( a := \varphi(1) \) and \( b := \varphi(-1) \) and writing 
\begin{equation}\label{eq:varphidef}
\varphi = \varphi_{a,b} + \tilde \varphi
\quad \text{ where } \quad
  \varphi_{a,b} \coloneqq b  \chi_{[-1,0)} + a \chi_{[0,1]}
\  \text{and}\  
  \tilde \varphi \coloneqq \varphi - \varphi_{a,b},
\end{equation}
where  \( \chi_{[-1,0)}, \chi_{[0,1]}\) are the characteristic functions of the intervals \( [-1,0) \) and \( (0,1]\) respectively. 
The induced observable   of \( \varphi \) is  the sum of the induced observables of \( \varphi_{a,b}\) and \(  \tilde \varphi  \) giving
\begin{equation}\label{eq:dec}
  \Phi (x)
  = \sum_{ k = 0 }^{ \tau (x) - 1 }  \varphi_{a,b} \circ g^{k} (x) +\sum_{ k = 0 }^{ \tau(x) - 1 } \tilde{ \varphi } \circ g^{k} (x) = \tau_{a,b} + \tilde \Phi, 
\end{equation}
where \(   \tilde{\Phi}  \) denote the induced observable of \( \tilde \varphi\), and \( \tau_{a,b}  \) is defined in \eqref{eq:def-of-tau}, indeed,  \( \varphi_{a,b}\circ g^k(x) \) takes only two possible values, \( a\) or \( b \),  depending on whether \( g^k(x)\in (0,1]\) or \( g^k(x)\in [-1,0)\), and therefore the corresponding induced observable is precisely \( \tau_{a,b} \). 

To prove  Theorem \ref{thm:limit-theorems} we obtain regularity and distribution results for the induced observables  \(  \tau_{a,b}\) and \(  \tilde \Phi  \) and substitute  them into   \eqref{eq:dec} to get the various cases  \eqref{eq:cond-clt}-\eqref{eq:cond-SL}. 
The motivation for the decomposition \eqref{eq:varphidef} is given by the observation that 
 \( \tilde \varphi (-1) = \tilde \varphi (1) = 0 \), which allows us to prove the following estimate for the corresponding induced observable \( \tilde \Phi\).

\begin{prop}\label{prop:bounds-on-tilde-Phi}
  Let \( g \in \mathfrak{F} \)  with  \( \beta \in (0,1) \) and let \( \tilde\varphi: [-1,1]\to \mathbb R\) be a H\"older continuous observable such that \( \tilde\varphi (-1)=\tilde\varphi (1) = 0\).   Then 
 \begin{equation}\label{eq:bounds1}
\text{  \ref{itm:H1}} \implies \tilde \Phi \in L^2 
\quand 
 \text{\ref{itm:H2}}  \implies \hat \mu( \pm \tilde \Phi > t ) = o( t^{ -1/\beta_\varphi }). 
 \end{equation}
\end{prop}


Proposition \ref{prop:tail-of-tau} gives results for \( \tau_{a,b}\). 

\begin{cor}[Corollary to Proposition \ref{prop:tail-of-tau}]\label{cor:tail-tau}
 If at least one of \( a \coloneqq \varphi(1) \), \( b \coloneqq \varphi(-1) \) is non-zero then:
 \[ 
   \beta_{\varphi} \in [0,1/2) \implies \tau_{a,b} \in L^2( \hat{\mu} ),
   \quand
   \beta_{\varphi} \in [1/2, 1) \implies \tau_{a,b} \in \mathcal{D}_{1/\beta_{\varphi}} . 
 \]
\end{cor}

We prove Corollary  \ref{cor:tail-tau} and   Proposition \ref{prop:bounds-on-tilde-Phi} in Section \ref{sec:proof-phi-in-Lq}. 
 For now we show how they imply Theorem \ref{thm:limit-theorems}. 
 
\begin{proof}[Proof of Theorem \ref{thm:limit-theorems}]
If  \( \varphi (-1) = \varphi (1) = 0 \)  then  \( \tau_{a,b} \equiv 0 \) and so \( \Phi = \tilde \Phi \), Proposition \ref{prop:bounds-on-tilde-Phi} implies that \( \Phi \in L^{2}( \hat \mu )\) and so \eqref{eq:cond-clt} holds.
 If at least one of \( \varphi (-1),  \varphi (1)  \) is non-zero, we have two cases.
If \( \beta_{\varphi} \in (0,1/2)\), Proposition \ref{prop:bounds-on-tilde-Phi} and Corollary \ref{cor:tail-tau}   give that both \( \tau_{a,b}, \tilde \Phi \in L^{2} ( \hat \mu )\), which implies that  \( \Phi\in L^{2} ( \hat \mu )\)  and therefore \eqref{eq:cond-clt} holds.
 If \( \beta_{\varphi} \in [1/2, 1) \) then  \( \tau_{a,b} \in \mathcal{D}_{1/\beta_{\varphi}}\) by Corollary \ref{cor:tail-tau} and \( \hat \mu ( \pm \tilde \Phi > t ) = o( t^{-1/\beta_{\varphi}} ) \) by Proposition \ref{prop:bounds-on-tilde-Phi}, and therefore   \( \Phi = \tau_{a,b} + \tilde{\Phi} \in \mathcal{D}_{1/\beta_{\varphi}}\) since the tail of \( \tilde \Phi \) is negligible compared to that of \( \tau_{a,b} \). Whence, \eqref{eq:cond-clt-ns} holds when \( \beta_{\varphi} = 1/2 \) and \eqref{eq:cond-SL} holds otherwise.
\end{proof}

\begin{rem} 
The relation between 
\( \Phi \) and \( \tau_{a,b} \) is given formally in \eqref{eq:dec} but it can be useful to have a heuristic idea of this relationships. Given a point \( x \in \delta_{i,j} \) with \( i,j \) both large we know that most of the first \( i \) iterates \( x, g(x), \dots, g^{ i - 1 } (x) \) will lie near the fixed point \( 1 \). Similarly, most of the next \( j \) iterates \( g^{ i } (x) , \dots, g^{ i + j - 1 } \) will lie near the fixed point \( -1 \). Thus, if we assume that \( \varphi \) is ``sufficiently well behaved'' near \( 1 \) and \( -1 \) (in a sense that is made precise by conditions \ref{itm:H1} and  \ref{itm:H2}), it is reasonable to hope that the induced observable \( \Phi \) at the point \( x \) will behave like \( \Phi (x) = \sum_{ k = 0 }^{ n - 1} \varphi \circ g^k \approx a i + b j = \tau_{a,b}(x) \) when  \( a= \varphi(1), b=\varphi(-1)\) are not both zero. 
\end{rem}

\section{The Induced Map}\label{sec:induced}

In this section we prove Proposition~\ref{thm:inducedmap}.  We begin by recalling one of several essentially equivalent definitions of Gibbs-Markov map. 
\begin{defn}
An interval map \( F: I \to I \) is called a (full branch) Gibbs-Markov map if there exists a partition \( \mathcal P \) of \( I \) (mod 0) into open subintervals such that: 
\begin{enumerate}
\item \( F \) is \emph{full branch}: for all \(\omega \in \mathcal P \) the restriction \(F|_{\omega}: \omega \to int(I) \) is a \( C^{1}\) diffeomorphism; 
\item \( F \) is \emph{uniformly expanding}: there exists \( \lambda > 1 \) such that \( |F'(x)|\geq \lambda\) for all \( x\in \omega \) for all \( \omega \in \mathcal P\);
\item \( F \) has \emph{bounded distortion}: 
there exists \( C>0, \theta \in (0,1) \) s.t. for all $\omega\in \mathcal P$  and  all $x, y\in\omega $, 
\[\log \bigg|\frac{F'(x)}{F'(y)}\bigg|\leq  C\theta^{s(x,y)}, 
\] 
where \(
  s ( x, y ) \coloneqq
  \inf \{ n \geq 0 : F^n x \text{ and } F^n y\) lie in different elements of the partition  \(  \mathscr{P} \}. 
\) 
\end{enumerate}
\end{defn}

We will show that the first return map \( G\) defined in \eqref{eq:G} satisfies all the conditions above as well as the saturation condition \eqref{eq:sat}. 
In  Section \ref{sec:top} we describe the topological structure of \( G \) and show that it is  a full branch map with countably many branches which saturates \( I \); this will require only the very basic topological structure of \( g \) provided by condition~\ref{itm:A0}. In Section~\ref{sec:est} we obtain estimates concerning the sizes  of the partition elements of the corresponding partition; this will require the explicit form of the map \( g \) as given in \ref{itm:A1}. In Section \ref{sec:exp} we show that \( G \) is uniformly expanding; this will require the final condition \ref{itm:A2}. Finally, in Section~\ref{sec:dist} we use the estimates and results obtained  to show that \(G \)  has bounded distortion.

\subsection{Topological Construction}\label{sec:top}

In this section we give an explicit and purely topological construction of the first return maps \( G^{-}: \Delta_0^- \to \Delta_0^-\)  and \( G^{+}: \Delta_0^- \to \Delta_0^-\)  which essentially depends only on condition \ref{itm:A0}, i.e.\ the fact that \( g \) is a full branch map with two orientation preserving branches. Recall first of all the definitions of the sets \( \Delta_{n}^{\pm}\) and \( \delta_{n}^{\pm}\) in \eqref{eq:Delta} and \eqref{eq:delta}. It follows immediately from the definitions and from the fact that each branch of \( g \) is a \( C^{2}\) diffeomorphism, that for every  \( n \geq 1 \),  the maps 
\(
 g:\delta_{n}^{-} \to  \Delta_{n-1}^{+}
\) 
and 
\( 
 g:\delta_{n}^{+} \to  \Delta_{n-1}^{-}
\) 
are \( C^{2}\) diffeomorphisms, and, for \( n \geq 2 \), the same is true for the maps 
\(
g^{n-1}: \Delta_{n-1}^{-}\to \Delta_0^{-}, 
\) and 
\( 
g^{n-1}: \Delta_{n-1}^{+}\to \Delta_0^{+},  
\) 
which implies that for every \( n \geq 1\), the maps 
\[
g^{n}: \delta_n^{-} \to \Delta_0^+ 
\quand 
g^{n}: \delta_n^{+} \to \Delta_0^- 
\]
 are  \( C^{2} \) diffeomorphisms. We can therefore define two maps 
\begin{equation}\label{def:Gtilde}
\widetilde G^-:\Delta_{0}^{-} \to \Delta_{0}^{+} \quand 
\widetilde G^+:\Delta_{0}^{+} \to \Delta_{0}^{-}
\quad \text{ by } \quad \widetilde G^\pm|_{\delta_{n}^{\pm} } :=  g^{n}. 
\end{equation}
Notice that these are  \emph{full branch} maps although they have \emph{different domains and ranges}, indeed the domain of one is the range of the other and viceversa. The fact that they are full branch allows us to \emph{pullback} the partition elements \( \delta_{n}^{\pm}\) \emph{into each other}: for every \(m, n \geq 1\) we let 
\[
\delta^-_{m,n} := g^{-m}(\delta^+_n) \cap \delta^-_m
\quand
\delta^+_{m,n} := g^{-m}(\delta^-_n) \cap \delta^+_m.
\] 
Then, for  \( m \geq 1\), the sets 
\(
 \{ \delta_{m,n}^{-}\}_{n\geq 1}  
\) 
and 
\( 
 \{ \delta_{m,n}^{+}\}_{n\geq 1} 
\)
are partitions of \( \delta_m^-\) and \( \delta_m^+\)  respectively and so 
\begin{equation}\label{eq:partitions}
\mathscr P^- :=  \{ \delta_{m,n}^{-}\}_{m,n\geq 1} 
 \quand 
 \mathscr P^+ :=  \{ \delta_{m,n}^{+}\}_{m,n\geq 1} 
\end{equation}
are partitions of \( \Delta_0^-, \Delta_0^+\)  respectively, with the property that for every \( m,n \geq 1\), the maps 
\begin{equation}\label{eq:fullbranch}
g^{m+n}: \delta_{m,n}^{-} \to \Delta_{0}^{-}
\quand 
g^{m+n}: \delta_{m,n}^{+} \to \Delta_{0}^{+}
\end{equation}
are \( C^2\) diffeomorphisms. Notice that  \( m+ n \) is the \emph{first return time} of points in \( \delta_{m,n}^{-} \) and \( \delta_{m,n}^{+} \)  to \( \Delta_{0}^{-} \) and  \( \Delta_{0}^{+} \) respectively and we have thus constructed two \emph{full branch first return induced maps} 
\begin{equation}\label{def:G-}
G^-:=\widetilde G^+ \circ \widetilde G^- :\Delta_{0}^{-} \to \Delta_{0}^{-} \quand G^+:=\widetilde G^- \circ \widetilde G^+ :\Delta_{0}^{+} \to \Delta_{0}^{+}.
\end{equation}
for which we have  
\(
G^-|_{\delta_{m,n}^{-} }= g^{m+n} 
\) and \( G^+|_{\delta_{m,n}^{+} }= g^{m+n}.  
\)

\begin{lem}\label{lem:fullbranch}
The maps \( G^{-}\) and \( G^{+}\) are full branch maps which saturate \( I \)
\end{lem}
\begin{proof}
The full branch property follows immediately from \eqref{eq:fullbranch}. It then also follows from the construction that   the families  
\[
\{g^j(\delta^-_{m,n})\}_{\substack{m,n\geq 1 \\  0\leq j < m+n}}
\quand 
\{g^j(\delta^+_{m,n})\}_{\substack{m,n\geq 1 \\  0\leq j < m+n}}
\]
of the images of the partition elements \eqref{eq:partitions} are each formed by a collection of \emph{pairwise disjoint} intervals which satisfy 
\[
\bigcup_{\delta^-_{m,n}\in \mathcal P^-} \bigcup_{j=0}^{m+n-1} g^j(\delta^-_{m,n})
= \bigcup_{\delta^+_{m,n}\in \mathcal P^+} \bigcup_{j=0}^{m+n-1} g^j(\delta^-_{m,n})
= I \mod 0
\]
and therefore  clearly satisfy \eqref{eq:sat}, giving the saturation.  
\end{proof}

\begin{rem}\label{rem:G}
Notice that the map \( G^{-}\) is exactly the first return map \( G \) defined in \eqref{eq:G} and therefore Lemma~\ref{lem:fullbranch} implies the first part of Proposition \ref{thm:inducedmap}.
\end{rem}


\subsection{Partition Estimates}\label{sec:est}
 
 The construction of the full branch induced maps \( G^{\pm}: \Delta_0^\pm \to \Delta_0^\pm\) in the previous section is purely topological and works for any map \( g \) satisfying condition \ref{itm:A0}. In this section we proceed to estimate the sizes and positions of the various intervals defined above, and this will require more information about the map, especially the forms of the map as given in \ref{itm:A1}. 
Before stating the estimates we introduce some notation. First of all, we let 
\( ( x_n^{-} )_{n \geq 0} \text{ and }  ( x_{n}^+ )_{ n \geq
 0} \) be the boundary points of the intervals \( \Delta_n^{-}, \Delta_n^{+}\) so that  
\(\Delta_0^{-}=(x_{0}^{-}, 0), \Delta_0^{+}=(0, x_{0}^{+})\) and, for every \( n \geq 1 
\) we have 
\begin{equation}\label{eq:xn}
\Delta_0^{-}=(x_{0}^{-}, 0), \qquad  \Delta_0^{+}=(0, x_{0}^{+}),
\qquad
\Delta_n^{-} = ( x_{n}^{-} , x_{n-1}^{-} ),
\qquad 
\Delta_n^{+} = ( x_{n-1}^{+} , x_n^{+}). 
\end{equation}
  The following proposition gives the speed at which the sequences \( (x_n^+), (x_n^-) \) converge to the fixed points \( 1, -1 \) respectively and gives estimates for the size of the partition elements \( \Delta_n^{\pm} \) for large \( n \) in terms of the values of \( \ell_1\) and \( \ell_2\).
  To state the result we  let 
    \[
  C_1 = \left(\ell_1 b_1\right)^{- {1}/{\ell_1}},
\quad
  C_2 = \left(\ell_2 b_2 \right)^{-{1}/{\ell_2}},
  \quad
  C_3 = \ell_1^{-(1 + 1/\ell_1)} b_1^{-1/\ell_1}
  \quad
  C_4 = \ell_2^{- (1 + 1/\ell_2)} b_2^{-1/\ell_2}.
    \]

\begin{prop}
    \label{prop:xn-and-Deltan}
    If $\ell_1  = 0$, then 
    \begin{equation} \label{eq:xn-0}
 (1+b_1+\epsilon)^{-n}\lesssim 1 + x_n^{-} \lesssim (1+b_1)^{-n}
\quand
     |\Delta_n^{-}| \lesssim  (1+b_1)^{-n}.
      \end{equation}
          If $\ell_1> 0$, then
    \begin{equation} \label{eq:xn-}
     1 + x_n^{-}  \sim C_1 n^{-{1}/{\ell_1}}
         \quand
 |\Delta_n^-| \sim C_3 n^{-\left(1 +{1}/{\ell_1}\right)}. 
   \end{equation}
          If $\ell_2 = 0$, then
    \begin{equation} \label{eq:xn+0}
  (1+b_2+\epsilon)^{-n}\lesssim 1 - x_n^{+} \lesssim  (1+b_2)^{-n}
\quand 
     |\Delta_n^{+}| \lesssim (1 + b_2)^{-n}.
      \end{equation}
  If $\ell_2 > 0$, then 
    \begin{equation} \label{eq:xn+}
      1- x_n^{+} \sim C_2 n^{- {1}/{\ell_2}}
         \quand
|\Delta_n^+| \sim C_4 n^{-\left(1 + {1}/{\ell_2}\right)}.
    \end{equation}
\end{prop}

\begin{proof}
We will  prove \eqref{eq:xn-0} and \eqref{eq:xn-} and then \eqref{eq:xn+0} and \eqref{eq:xn+} follow by exactly the same arguments.  Notice first of all that  from \eqref{eq:n+-} we have \( \delta^{+}_{n} \subset U_{0+}\) for all \( n > n_{+}\) and, since from \eqref{eq:U} we have that \(   U_{-1}:=g(U_{0+}) \),  this implies that   \( \Delta_{n-1}^{-}:=g(\delta_{n}^{+}) \subset U_{-1}\) for all \( n > n_{+}\), and  thus \( \Delta_{n}^{-} \subset U_{-1}\) for all \( n \geq n_{+}\) which, by the definition of \( x_{n}\) in \eqref{eq:xn}, implies that  \( x_n^{-}\in U_{-1}\) for \( n \geq n_+\).

Now suppose  that \( \ell_1>0\). For \( n \geq n_+\), by  definition of the \( x_{n}^{-} \), we have that $g(x_{n+1}^{-}) = x_n^{-}$, and so $1 + x_{n}^{-} = 1 + x_{n+1}^{-} + b_1 ( 1 + x_{n+1}^- )^{1 + \ell_1}$. Setting \( z_n = 1 + x_n^{-} \) we can write this as \(     z_n = z_{n+1} (1 + b_1 z_{n+1}^{\ell_{1}})\) and, 
 taking the power \( -\ell_{1} \) 
 and expanding we get 
  \begin{align*}
      \frac{1}{z_n^{\ell_{1}}} &= \frac{1}{z_{n+1}^{\ell_{1}}} (1 + b_1 z_{n+1}^{\ell_{1}})^{-\ell_{1}} = \frac{1}{z_{n+1}^{\ell_{1}}} \left( 1 - \ell_{1} b_1z_{n+1}^{\ell_{1}} + O(z_{n+1}^{2\ell_1})\right) = \frac{1}{z_{n+1}^{\ell_{1}}} - \ell_{1}b_1 + o(1).
  \end{align*}
  From the above we know that \( z_n^{-\ell_1} = z_{n-1}^{-\ell_1} + b_1\ell_1 + o(1)\) and applying this relation recursively 
  we obtain that \( z_{n}^{-\ell_1} = \ell_1 b_1 n + o (n) \) 
  which  yields \( x_n^{-} + 1 = ( \ell_1 b_1 n )^{ -{1}/{\ell_1} } (1 + o (1) ) \), thus giving the first statement in \eqref{eq:xn-}.  Now, by definition $\Delta_n^{-} = [ x_{n}^-, x_{n-1}^- ) = [ x_{n}^-, g(x_n^-) )$, so, for all \(n\) large enough, $|\Delta_n^{-}| = g(x_n^-) - x_n^{-} = b_1 (1 + x_n^{-} )^{ 1 + \ell_1}$. Inserting \( x_n^- + 1 \sim C_1 n^{ - 1 / \ell_1 } \) into this expression for $|\Delta_n^-|$ then yields \( |\Delta_n^-| \sim \ell_1^{-(1 + {1}/{\ell_1})} b_1 n^{ -1 - 1/\ell_1} \), completing the proof of \eqref{eq:xn-}.

  Now, for \( \ell_1=0\), since $g(x_n^-)=x_{n-1}^-$, the mean value theorem implies
  \((1+b_1)\leq (x_{n-1}^-+1)/(x_{n}^-+1)\leq (1+b_1+o(1))\) which can be written as \((1+b_1+o(1))^{-1}(x_{n-1}^-+1)\leq x_{n}^-+1\leq(1+b_1)^{-1}(x_{n-1}^-+1)\). Iterating this relation we obtain the claimed bounds for \( x_n^{-} + 1 \). As in the previous case we may calculate using \eqref{eq:A1b} that
  \(
    |\Delta_n^-| = g(x_{n}^{-}) - x_{n}^{-} = -1 + (1 + b_1) ( 1 + x_n^{-} ) + \xi ( x_n^{-} ) - x_{n-1} = b_1 ( 1 + x_n^{-} ) + o(1) \lesssim (1 + b_1 )^{-n}
  \), which concludes the proof.

\end{proof}


%
%

 To get analogous estimates for the intervals \( \delta_n^{-}, \delta_n^{+} \),  we let \( (y_n^{-})_{ n \geq 0 } \) and \( (y_n^{+} )_{ n \geq 0 } \) be the boundary points of the intervals  \( \delta_n^{-}, \delta_n^{+} \) respectively, so that for every \( n\geq 1 \)  we have 
\[
\delta_n^{-} = ( y_{n-1}^{-}, y_{n}^- ) \quand 
 \delta_n^{+} = ( y_{n}^{+}, y_{n-1}^{+}).
 \]  
In particular, 
  \( y_{0}^{-} = x_0^-\), \( y_{0}^{+} = x_0^+\), and \( g(y^-_n) = x^+_{n-1}\), \( g(y^+_n) = x^-_{n-1}\) for \( n \geq 1 \).
 Then we let
  \[
    B_1 = a_1^{-1/k_1} (\ell_2 b_2)^{-1/\beta_{2}},
    \quad
 B_2 = a_2^{-1/k_2} (\ell_1 b_1)^{-1/\beta_{1}}, 
\quad
    B_3 = B_1/\beta_{2},
\quad
    B_4 = B_2/\beta_{1}.
    \]
Recall that \( B_{1}, B_{2}\) have already been defined in \eqref{eq:B}.  
\begin{prop}
    \label{prop:y_n-and-delta_n}
       If $\ell_1= 0$, then for every \( \eps > 0 \)
   \begin{equation}\label{eq:yn-0}
    \left(\frac{1}{1+b_2+\varepsilon}\right)^{\frac{n}{k_1}} \lesssim -y_n^{-} \lesssim \left(\frac{1}{1+b_2}\right)^{\frac{n}{k_1}}
\quand
  |\delta_n^{-}| \lesssim \left(\frac{1}{1+b_2}\right)^{n}.
\end{equation}
    If $\ell_1> 0$, then
        \begin{equation}
        \label{eq:yn1}
         y_n^{-} \sim - B_1 n^{-\frac{1}{\beta_{2}}}, 
  \quand
        |\delta_n^{-}| \sim B_3 n^{-\left( 1 + \frac{1}{\beta_{2}} \right)}.
          \end{equation}
         If $\ell_2 = 0$, then for every \( \eps > 0 \)
     \begin{equation}\label{eq:yn+0}
    \left(\frac{1}{1+b_1+\varepsilon}\right)^{\frac{n}{k_2}} \lesssim y_n^{+} \lesssim \left(\frac{1}{1+b_1}\right)^{\frac{n}{k_2}},
\quand
  |\delta_n^{+}| \lesssim \left(\frac{1}{1+b_1}\right)^{n}.
    \end{equation} 
    If $ \ell_2 > 0$, then
        \begin{equation}
        \label{eq:yn2}
        y_n^{+} \sim B_2 n^{-\frac{1}{\beta_{1}}}, 
\quand
         |\delta_n^{+}| \sim B_4 n^{-\left( 1 + \frac{1}{\beta_{1}} \right)},
    \end{equation}
\end{prop}

\begin{proof}
We will prove \eqref{eq:yn-0} and \eqref{eq:yn1}, as \eqref{eq:yn+0} and \eqref{eq:yn2}   follow by  analogous arguments. 
Suppose first  that \( \ell_1 > 0 \). 
  As \(x_n^+ \to 1\), and as
  \(g_{-}(y_n^-) = x_{n-1}^{+}\) we know that for all \(n\) sufficiently large we have \(g_{-} (y_n^-) = 1 - a_1 (-y_n^-)^{k_1} = x_{n-1}^{+}\).  Solving for \( y_n \) this gives
  \[
    y_{n}^{-} = - \left( { (1 - x_{n-1}^{+}) }/{a_1} \right)^{1/k_1}
      = - a_1^{-1/k_1} \left(\ell_2 b_2 n\right)^{-{1}/{\ell_2 k_1}} (1 + o(1))
  \]
which is the first statement in \eqref{eq:yn1}. 
  Now we turn our attention to the size of the intervals \(\delta_n^{-}\). First let us note that for any \(\gamma > 0\) we have that 
  \( 
        n^{-\gamma} - (n + 1)^{-\gamma}
      = n^{-\gamma} \left[1 - (1 + 1/n)^{-\gamma}\right] 
      = n^{-\gamma} \left[ 1 - \left(1 - \gamma/{n} + O(n^{-2})\right) \right]
      = \gamma n^{-(1 + \gamma)} (1 + O(1/n))
  \) 
and therefore 
  \[
        |\delta_n| =    y_n^{-} - y_{n+1}^{-} = B_1 ( n^{-1/\ell_2 k_1} - (n+1)^{-1/\ell_2 k_1}) (1 + o(1)) =  \frac{B_1}{\ell_2 k_1} n^{-(1 + 1/\ell_2 k_1)}(1 + o(1))
  \]
which completes the proof of  \eqref{eq:yn1}. 
  Now for \(\ell_2=0\) we proceed as before, and  by \eqref{eq:xn+0} we get 
  \[
      (1 + b_2 + \eps )^{-n/k_1} \lesssim  - y_n^-  =  \left(  ({ 1 - x_{n-1}^{+} })/{a_1} \right)^{1/k_1} \lesssim (1 + b_2)^{-n/k_1}. 
      \]
  For the size of the interval \( \delta_n^- \), we may use the mean value theorem to conclude that
  \[
    g'(u_n) = \frac{ x_{n-2}^+ - x_{n-1}^+ }{ y_{n-1}^{-} - y_{n}^{-} } = \frac{ | \Delta_{n-1 }^+|}{ |\delta_n^-| },
  \]
  for some \( u_n \in \delta_n^- \). As \( g' \) is monotone on \( U_{0}^- \) we know, from the above and \eqref{eq:xn+0}, that
  \[
    |\delta_n^-| \lesssim  {|\Delta_{n-1}^+|}/{g'(y_n^{-})} \lesssim (1 + b_2 + \eps )^{-n/k_1} ( 1 + b_2 )^{-n}
    \lesssim  ( 1 + b_2 )^{-n}.
  \]
  which concludes the proof.
\end{proof}

\subsection{Expansion Estimates}\label{sec:exp}

\begin{prop}\label{prop:expansion-new}
  For every \( g \in \widehat{\mathfrak F} \) the first return map \( G : \Delta_0^- \to \Delta_0^- \) is uniformly expanding.
\end{prop}

It is enough to prove uniform expansivity for the two maps \( \widetilde G^-, \widetilde G^+\), recall \eqref{def:Gtilde}, since this implies the same property for their composition \( G= G^{-}\), recall \eqref{def:G-}.
To simplify the notation we will only prove the statement for \( \widetilde G^+\), i.e. we will prove that \( x \in \delta_{n}^+ \Rightarrow (g^n)'(x) > \lambda \). The fact that \( x \in \delta_{n}^- \Rightarrow (g^n)'(x) > \lambda \) follows by an identical argument. 

For points outside the neighourhood \( U_{0+}\) on which the map \( g \) has a precise form, more precisely for \( 1\leq n \leq n_+\) and for  \( x \in \delta_{n}^+ \), the expansivity is automatically guaranteed by condition \ref{itm:A2},  but for points close to 0 where the derivative can be arbitrarily small the statement is  non-trivial.  It  ultimately depends  on writing \( \widetilde G^+(x) := g^n(x) \) for \( x \in \delta_{n}^+ \), so that  \( ( \widetilde G^+)'(x) = (g^n)'(x) = (g^{n-1})'(g(x)) g'(x)\), and then showing  that the, potentially small derivative \( g'(x)\) near \( 0 \) is compensated by sufficiently large number of iterates where the derivative is \( > 1 \). This clearly relies very much on the partition estimates in Section \ref{sec:est} which provide a relation between the position of points, and therefore their derivatives, and the corresponding values of \( n \). A  relatively straightforward computation using those estimates shows that we get expansion for \emph{sufficiently large} \( n \geq 1 \), which is quite remarkable but not enough for our purposes as it does not give a complete proof of expansivity for \( \widetilde G^+\) at every point in \( \Delta_0^+\). We therefore need to use a somewhat more sophisticated approach that shows that the derivative of  \( \widetilde G^+\) has a kind of \emph{``monotonicity''} property in the following sense. 
Define the function \( \phi : \Delta_0^+ \setminus \delta_1^+ \to \Delta_0^+ \) given implicitly by \( g^2 = g \circ \phi \) and explicitly by 
\begin{equation}\label{eq:defphi}
\phi \coloneqq (g|_{U_{0+}})^{-1} \circ g|_{U_{-1}} \circ g|_{U_{0+}}
\end{equation}
Notice that  \( \phi \) is the bijection which makes the diagram in Figure~\ref{fig:def-of-phi} commute.
\begin{figure}[h]
  \centering
  \begin{tikzcd}
      \delta_{n + 1}^+  \arrow[r, "\phi"] \arrow[d, "g"] & 
      \delta_{n}^+  \arrow[d, "g"] &
       \\
      \Delta_{n}^-  \arrow[r, "g"] &
      \Delta_{n-1}^-  & 
  \end{tikzcd}
  \caption{\label{fig:def-of-phi} Definition of the map \( \phi : \Delta_{0}^+\setminus \delta_{1}^+ \to \Delta_{0}^+ \).}
\end{figure}

%

The key step in the proof of Proposition \ref{prop:expansion-new} is  the following lemma.

\begin{lem}
  \label{lem:expansion}
For all  \( n \geq n^+\)  and  \( x \in \delta_{n+1}^+\) we have 
  \[
    (g^{2})'(x) > g'(\phi(x)).
  \]
\end{lem}

\begin{rem}
Lemma \ref{lem:expansion} is equivalent to \( {(g^2)' (x)}/{g'(\phi(x))} > 1\)  
which is equivalent to 
\begin{equation}\label{eq:phiexpansion}
 \frac{ g'(x) }{ g'(\phi(x))} g'(g(x)) > 1. 
\end{equation}
If \( k_2 \in (0,1) \) then we know that \( g' \) is monotone decreasing on \( U_{0}^+ \) and so, as \( x < \phi(x) \), we have that \( g ' ( x ) / g ' ( \phi (x ) ) > 1 \). By construction \( g(x) \in U_{1} \) for every \( x \in U_{0}^+ \), so \( g' ( g (x) ) > 1 \) which concludes that \eqref{eq:phiexpansion} holds.

On the other hand, if \( k_2 > 1 \), then the ratio \({ g'(x) }/{ g'(\phi(x)) } \)  is \( < 1 \) and measures how much derivative is ``lost'' when choosing the initial condition  \( x \) instead of the initial condition \( \phi(x) \) (since \(  \phi(x) > x\) and the derivative is monotone increasing), whereas \( g'(g(x)) > 1 \) measures  how much derivative is ``gained'' from performing an extra iteration of \( g \). The Lemma says that the gain is more than the loss. 
\end{rem}

\begin{proof}
In light of the remark above we will assume that \( k_2 > 1 \). To simplify the notation let us set \( a = a_{2} \), \( b = b_1 \), \( k = k_2 \), and \( \ell = \ell_1 \). 
Notice first of all  that by the form of \( g \) in \( U_{0+}\) given in \ref{itm:A1}  we have 
  \begin{equation}\label{eq:phix}
     \frac{g'(x)}{g'(\phi(x))} = \left(\frac{ x }{ \phi (x) }\right)^{k - 1} = 
     \left(\frac{ \phi (x) }{ x }\right)^{1-k}
  \end{equation}
Recall that \( k > 1\) and \( x< \phi(x)\)  and so the ratio above is \( < 1\). To estimate \( g'(g(x))\) we consider two  cases depending on  \( \ell\). If  \( \ell > 0 \), using the form of \( g \) given in \ref{itm:A1} and plugging into \eqref{eq:defphi} we get 
\[
    \phi (x) =  \left[ x^{k} + b a^\ell x^{ k ( \ell + 1 )} \right]^{{1}/{k}}
 \quad  \text{  and therefore }
 \quad 
\left( \frac{\phi (x)}{x} \right)^k  = 1 + b a^\ell x^{ k \ell }
\]
 and, therefore,  using the form of \( G \) in \( U_{-1}\), this gives 
  \begin{equation} \label{eq:def-phi-on-delta_n+1}
    g'(g(x)) = g'( -1 + a x^k )  = 1 + b a^\ell x^{k \ell} + b \ell a^{\ell} x ^{k \ell} =  \left( \frac{\phi (x)}{x} \right)^k + b\ell a^\ell x^{k \ell}.
  \end{equation}
 From \eqref{eq:phix} and \eqref{eq:def-phi-on-delta_n+1} and the fact that \( x< \phi(x)\) we immediately get 
  \[
    \frac{g'(x)}{g'(\phi(x))} g'(g(x)) 
    =    \left( \frac{\phi(x)}{x} \right)^{1-k}
     \left[ \left( \frac{ \phi(x)} { x} \right)^k  + b \ell a^{\ell} x^{ k \ell} \right] > 
 \frac{\phi(x)}{x} > 1
  \]
  which establishes \eqref{eq:phiexpansion} and completes the case that \( \ell > 0 \).
For  \( \ell = 0 \), proceeding as above we obtain
  \begin{equation}\label{eq:phi-when-ell-0}
    \phi (x) = \left[ (1 + b)x^k + \xi(g(x))/a \right]^{1/k}
     \quad  \text{  and therefore }
 \quad 
\left( \frac{\phi (x)}{x} \right)^k  = (1 + b) + \frac{\xi(g(x))}{ax^k}.
  \end{equation}
Since \( g(x) = -1+ax^k\),  from \eqref{eq:xider} we have \( \xi'(g(x)) \geq 
    \xi(g(x))/(1+g(x))=
  \xi(g(x))/ax^k)\), and so 
    \[
    g'(g(x)) = (1 + b) + \xi'(g(x)) \geq (1 + b) + \frac{\xi(g(x))}{ax^k} = \left( \frac{\phi (x)}{x} \right)^k. 
    \]
    Together with \eqref{eq:phix}, as above, we get the statement in this case also. 
\end{proof}

As an almost immediate consequence of Lemma \ref{lem:expansion} we get the following. 
\begin{cor}\label{cor:gm}
 For all  \( n \geq n^+\)  and   \( x \in \delta_{n+1}^+ \) we have 
  \[
(\widetilde G^{+}){'}(x) > (\widetilde G^{+}){'}(\phi (x)).   \]
\end{cor}

\begin{proof}
By Lemma \ref{lem:expansion} and \eqref{eq:phiexpansion},   for any \( 1 \leq m \leq n \) 
  we have 
  \begin{equation}
    (g^{m+1}) ' (x) = g'(x) g'(g(x)) \cdots g'(g^{m} (x))
    = \frac{ g'(x) g'(g (x))}{ g'( \phi(x)) }  (g^{m})'(\phi (x))
    > (g^{m})'(\phi (x)). 
  \end{equation}
\end{proof}

\begin{proof}[Proof of Proposition~\ref{prop:expansion-new}]
Condition \ref{itm:A2} implies that \( (\widetilde G^{+}){'}(x)\geq \lambda \)  for all  \( x\in \delta^{+}_{n}\) for \( 1\leq n \leq n^{+}\). 
Then, for \( x \in \delta^{+}_{n^{+}+1}\) we have \( \phi(x) \in \delta^{+}_{n^{+}}\) 
and therefore 
\[
(\widetilde G^{+}){'}(x) > (\widetilde G^{+}){'}(\phi (x))  \geq \lambda 
\]
Proceeding inductively we obtain the result. 
\end{proof}

\subsection{Distortion Estimates}
\label{sec:dist}
\begin{prop}\label{prop:boundeddist}
For all \( g \in \widehat{ \mathfrak{F} } \) there exists a constant \( \mathfrak D>0\) such that for all \( 0\leq m < n \) and all \( x, y\in \delta^{\pm}_n\),
\[
\log \dfrac{\left(g^{n-m}\right)'(g^m(x))}{\left(g^{n-m}\right)'(g^m(y))}
\leq  \mathfrak D |g^n(x) - g^n(y)|.
\]
\end{prop}

As a consequence we get that \( G \) is a Gibbs-Markov map with constants $C= \mathfrak D\lambda$ and $\theta=\lambda^{-1}$. 

\begin{cor}
  \label{cor:bounded-distortion}
For all $x, y\in\delta_{i,j}\in \mathcal P$ with  \( x \neq y\) we have 
\[\log \bigg|\frac{G'(x)}{G'(y)}\bigg|\leq  \mathfrak D  \lambda^{-s(x,y)+1}.
\] 
\end{cor}

\begin{proof}
 Let \( n \coloneqq s(x,y)\). 
Since  $G$ is uniformly expanding, we have \(1\geq |G^{n}(x)-G^{n}(y)|=|(G^{n-1})'(u)||G(x)-G(y)|\geq \lambda^{n-1}|G(x)-G(y)| \) and therefore  \(|G(x)-G(y)|\leq \lambda^{-n+1}.\)  By Proposition~\ref{prop:boundeddist} this gives 
\( 
\log |G'(x)/G'(y)|
\leq \mathfrak D|G(x)-G(y)|
\leq \mathfrak D \lambda^{-n+1} = \mathfrak D  \lambda^{-s(x,y)+1}
\). 
\end{proof}

\begin{proof}[Proof of Proposition \ref{prop:boundeddist}]
We begin with a couple of simple formal steps. First of all, by the chain rule, we can write 
\[
\log\dfrac{\left(g^{n-m}\right)'(g^m(x))}{\left(g^{n-m}\right)'(g^m(y))}=\log\prod_{i=m}^{n-1}\dfrac{g'(g^i(x))}{g'(g^i(y))}
=\sum_{i=m}^{n-1}\log\dfrac{g'(g^i(x))}{g'(g^i(y))}.
\]
Then, since \( g^i(x), g^i(y)\) are both in the same smoothness component of   \(g \),  by the Mean Value Theorem, there exists \( u_i\in (g^i(x),g^i(y))\) such that 
\[
\log \dfrac{g'(g^i(x))}{g'(g^i(y))} 
=  \log g'(g^i(x))-\log g'(g^i(y)) = \dfrac{g''(u_i)}{g'(u_i)} |g^i(x)-g^i(y)|.
\]
Substituting this into the expression above, and writing \( \mathfrak D_i:= {g''(u_i)}/{g'(u_i)}\) for simplicity, we get 
\begin{equation}\label{eq:dist1}
\log\dfrac{\left(g^{n-m}\right)'(g^m(x))}{\left(g^{n-m}\right)'(g^m(y))}
= \sum_{i=m}^{n-1} \mathfrak D_i |g^i(x)-g^i(y)| 
\leq \sum_{i=0}^{n-1} \mathfrak D_i |g^i(x)-g^i(y)|.
\end{equation}
We will bound the sum above in two steps. First of all we will show that it admits a uniform bound \( \widehat{\mathfrak D} \) independent of \( m,n\).  We will then use this bound to improve our estimates and show that by paying a small price (increasing the uniform bound to a larger bound \( \mathfrak D:= \widehat{\mathfrak D}^2/{|\Delta_0^-|} \)) we can include the term \( |g^n(x)-g^n(y)| \) as required. Ultimately this gives a stronger result since it takes into account the closeness of the points \( x,y\). 

Let us suppose first for simplicity that \( x, y\in \delta^{+}_n\), the estimates for \( \delta^{-}_n \) are identical.
Then for \( 1\leq i < n \) we have that  \( g^i(x), g^i(y), u_i\in \Delta^{-}_{n-i}\) and therefore we can bound \eqref{eq:dist1} by  
\begin{equation}\label{eq:dist2}
\sum_{i=0}^{n-1} \mathfrak D_i |g^i(x)-g^i(y)|
\leq \mathfrak D_0 |x-y| + \sum_{i=1}^{n-1} \mathfrak D_i |g^i(x)-g^i(y)|
\leq \mathfrak D_0 \delta^{+}_n + \sum_{i=1}^{n-1} \mathfrak D_i |\Delta^{-}_{n-i}|.
\end{equation}
From \eqref{deq_55} and using the relationship between the \( y_n^+ \) and the \( x_n^- \) we may bound the first term by
\begin{equation}\label{eq:di1}
\mathfrak D_0 |\delta_n^{+}|
\lesssim  u_0^{-1}  |\delta_n^{+}| 
\lesssim \frac{ y_n^+ - y_{n+1}^+}{y_{n+1}^+}  \lesssim \left( \frac{ 1 + x_n^- }{1 + x_{n+1}^{-}}\right)^{1/k} - 1 \to c < \infty
\end{equation}
where we have used the fact that  that for some sequence \( \xi_n \to -1 \) we have  \( { (1 + x_n^- )}/{(1 + x_{n+1}^{-})} = g' ( \xi_n ) \) which converges to \( 1 \) if \( \ell > 0 \) (and therefore \( c = 0 \)) or \( 1 + b_1 \) otherwise (and therefore \( c = b_{1} \)).
If \( \ell_1 = 0 \) then \( \mathfrak D_i \) is uniformly bounded for \( i > 0 \), if \( \ell_1 > 0 \), then from \eqref{eq:2ndder1} and \eqref{eq:xn-} we know that
\begin{equation}\label{eq:di2}
\mathfrak D_i |\Delta_{n-i}^-| 
\lesssim (1+u_i)^{\ell_1-1} |\Delta_{n-i}^-| 
\lesssim  (n-i)^{-\frac{\ell_1-1}{\ell_1}} (n-i)^{-\left(1 + \frac{1}{\ell_1}\right)} 
= (n-i)^{-2}.
\end{equation}
Then by \eqref{eq:di1} and \eqref{eq:di2}  we find that \
\begin{equation}\label{eq:dist3}
\widehat{ \mathfrak D } \coloneqq \exp \left\{ \mathfrak D_0 \delta_{n}^+ + \sum_{i = 1}^\infty \mathfrak D_i \Delta_{ n - i } \right\} \leq \exp \left\{  \mathfrak D_0  + \sum_{i = 1}^\infty \mathfrak D_i \Delta_{ n - i } \right\} < \infty.
\end{equation}
 Substituting this back into \eqref{eq:dist2} and then into \eqref{eq:dist1} we get

\begin{equation}\label{eq:dist4}
 \log\bigg|\dfrac{\left(g^{n-m}\right)'(g^m(x))}{\left(g^{n-m}\right)'(g^m(y))}\bigg|
\leq \log \widehat{\mathfrak D}
\end{equation}
which completes the first step in the proof, as discussed above. We now take advantage of this bound to improve our estimates as follows. By a standard and straightforward application of the Mean Value Theorem, \eqref{eq:dist4} implies that the diffemorphisms \( g^{n}: \delta_{n}^+ \to \Delta_0^-\) and \( g^{n-m}: \Delta_{n-m}^- \to \Delta_0^-\) all have uniformly bounded distortion in the sense that for every \( x,y \in \delta^+_n\) and \( 1\leq m < n \) we have
\begin{align}\label{dist1}
\dfrac{|x-y|}{|\delta_{n}^+|}
\leq \widehat{\mathfrak D}
\dfrac{|g^n(x)-g^n(y)|}{|\Delta_0^-|}
\end{align}
and 
\begin{align}\label{dist1'}
\dfrac{|g^m(x)-g^m(y)|}{|\Delta_{n-m}^+|}
\leq \widehat{\mathfrak D}
\dfrac{|g^{n-m}(g^m(x))-g^{n-m}(g^m(y))|}{|\Delta_0^-|}=\widehat{\mathfrak D}
\dfrac{|g^n(x))-g^n(y)|}{|\Delta_0^-|}.
\end{align}
Therefore 
\[
|x-y|
\leq  \frac{\widehat{\mathfrak D}}{|\Delta_0^-|} 
|g^n(x)-g^n(y)|{|\delta_{n}^+|}
\quand 
|g^m(x)-g^m(y)|
\leq  \frac{\widehat{\mathfrak D}}{|\Delta_0^-|} 
|g^n(x)-g^n(y)|
{|\Delta_{n-m}^-|}.
\] 
Substituting these bounds back into \eqref{eq:dist1} (with \( i=m\)), and letting \( \mathfrak D:= \widehat{\mathfrak D}^2/{|\Delta_0^-|} \), we get 
\[
\begin{aligned}
\log\frac{\left(g^{n-m}\right)'(g^m(x))}{\left(g^{n-m}\right)'(g^m(y))}
& \leq \sum_{i=0}^{n-1} \mathfrak D_i |g^i(x)-g^i(y)| 
= \mathfrak D_0 |x-y| + \sum_{i=1}^{n-1} \mathfrak D_i |g^i(x)-g^i(y)| 
\\ & 
\leq \mathfrak D_0   \frac{\widehat{\mathfrak D}}{|\Delta_0^-|} 
|g^n(x)-g^n(y)|  {|\delta_{n}^+|}
+ \sum_{i=1}^{n-1} \mathfrak D_i \frac{\widehat{\mathfrak D}}{|\Delta_0^-|} 
|g^n(x)-g^n(y)|{|\Delta_{n-i}^-|}
\\ & 
=  \frac{\widehat{\mathfrak D}}{|\Delta_0^-|} 
 \left[
 \mathfrak D_0  
 {|\delta_{n}^+|}
+ \sum_{i=1}^{n-1} \mathfrak D_i 
{|\Delta_{n-i}^-|}
\right] |g^n(x)-g^n(y)|
\\ &
\leq 
\frac{\widehat{\mathfrak D}^2}{|\Delta_0^-|}  |g^n(x)-g^n(y)| 
= \mathfrak D |g^n(x)-g^n(y)|.
\end{aligned}
\]
Notice that the last inequality follows from \eqref{eq:dist3}. This completes the proof. 
\end{proof}

We state here also a simple corollary of Propositions \ref{prop:y_n-and-delta_n}  and \ref{prop:boundeddist} which we will use in Section~\ref{sec:statistical}.

\begin{lem}\label{lem:estimates-on-delta-ij}
For all \( i,j \geq 1\)  we have 
  \begin{equation}\label{eq:estimates-delta-ij}
  \tilde \mu( \delta_{i,j} ) \lesssim
    \begin{cases}
      i^{ -(1 + 1/\beta_{2}) } j^{ -( 1 + 1/\beta_{1})},&\text{if } \ell_1, \ell_2 > 0 \\
      (1 + b_2)^{-i}, &\text{if } \ell_1 =0 , \ell_2 > 0  \\
      (1 + b_1)^{-j} , &\text{if } \ell_2 =0 , \ell_1 > 0  \\
      \min\{ (1 + b_1), (1 + b_2) \}^{-i -j}, &\text{if } \ell_1 =0 , \ell_2 = 0
    \end{cases}
  \end{equation}
\end{lem}
\begin{proof}
Proposition \ref{prop:boundeddist} implies that \(  |\delta_{ij}| \approx |\delta_{i}^{-}| |\delta_{j}^{+}| \) 
uniformly for all \( i,j \geq 1\).  Indeed, more precisely, it implies 
\(
\mathfrak D^{-1} {|g^{i}(\delta_{ij})|}/{|g^{i}(\delta_{i})|} \leq  {|\delta_{ij}|}/{|\delta_{i}|} \leq 
\mathfrak D  {|g^{i}(\delta_{ij})|}/{|g^{i}(\delta_{i})| }
\)
which implies 
\( 
|\delta_{i}^{-}| |\delta_{j}^{+}|/\mathfrak D |\Delta_{0}^{+}|  \leq |\delta_{ij}| 
\leq\mathfrak D |\delta_{i}^{-}| |\delta_{j}^{+}|/ |\Delta_{0}^{+}|.
\) 
As \( \tilde \mu \) is equivalent to Lebesgue on \( \Delta_0^-\cup \Delta_{0}^{+} \) we obtain the Lemma immediately from   Proposition \ref{prop:y_n-and-delta_n}.
\end{proof}

%
%

\section{Statistical Properties}\label{sec:statistical} 





In Section \ref{sec:tail} we prove Proposition \ref{prop:tail-of-tau} and in Section~\ref{sec:proof-phi-in-Lq} we prove Proposition~\ref{prop:bounds-on-tilde-Phi} and Corollary~\ref{cor:tail-tau}.  As discussed in Section \ref{sec:distind} this completes the proof of  Theorem~\ref{thm:limit-theorems}.

\subsection{Distribution and Tail Estimates}\label{sec:tail}

In this section we prove Proposition \ref{prop:tail-of-tau}.
We will only explicitly prove~\eqref{eq:a+b-tail} and~\eqref{eq:a-b-pstve-tail} as the proof of~\eqref{eq:a-b-neg-tail} is identical to that of~\eqref{eq:a-b-pstve-tail}. For \( a, b \geq 0 \) consider the following decompositions
\begin{align}
  \tilde \mu( a\tau^{+} + b \tau^{-} > t ) 
  =& \tilde \mu( a\tau^{+} > t) + \tilde \mu( b\tau^{-} > t)  \label{eq:tail-of-Phi-positive-case1}  \\
  \label{eq:tail-of-Phi-positive-case}  
  &- \tilde \mu ( a\tau^{+} > t, b\tau^{-}  > t) + \tilde \mu ( a\tau^{+} + b \tau^{-} > t, \max\{a\tau^{+}, b\tau^{-}\} \leq t )
\end{align}
and
\begin{align} 
  \label{eq:tail-of-Phi-mixed-case}
  \tilde\mu ( a \tau^+ - b\tau^- > t  ) =&  \tilde \mu ( a\tau^{+} > t) 
  \\
  &- \tilde \mu( a\tau^{+} > t, a\tau^{+} - b \tau^{-} \leq t ).
  \label{eq:tail-of-Phi-mixed-case1}
\end{align}

We can then reduce  the proof to two further  Propositions.  
First of all we give precise asymptotic estimates of the terms \( \mu ( a \tau^+ > t ) \), \( \mu ( b \tau^- > t ) \) which make up \eqref{eq:tail-of-Phi-positive-case1} and \eqref{eq:tail-of-Phi-mixed-case}.

\begin{prop}\label{lem:dist-of-atau-and-btau}
    For every \( a,b \geq 0 \) and for every \( \gamma \in ( 0, 1 ) \)
    \begin{equation}\label{eq:dist-of-tau+}
        \tilde\mu ( a\tau^+ > t ) = 
        \begin{cases} 
            C_a t^{-1/\beta_2} + o( t^{-\gamma - 1/\beta_2} ) & \text{if } \ell_1 > 0  \\
            O \left( ( 1 + b_2 )^{-t/ak_1} \right)& \text{if } \ell_1 = 0
        \end{cases}
    \end{equation}
    and 
    \begin{equation}\label{eq:dist-of-tau-}
        \tilde \mu ( b\tau^- > t ) = 
        \begin{cases} 
             C_b t^{-1/\beta_1} + o( t^{-\gamma - 1/\beta_1} ) & \text{if } \ell_2 > 0  \\
            O \left( ( 1 + b_1 )^{-t/bk_2} \right)& \text{if } \ell_2 = 0
        \end{cases}
    \end{equation}
\end{prop}
Then, we  show that the remaining terms \eqref{eq:tail-of-Phi-positive-case} and \eqref{eq:tail-of-Phi-mixed-case1} in the decompositions above have negligible contribution to the leading order asymptotics of the tail.

\begin{prop}
\label{prop:higher-order}
 If at least one of \( \ell_1, \ell_2 \) are not zero,  then for every \( a, b \geq 0 \), \( \gamma \in ( 0, 1 ) \) we have 
  \begin{equation}
    \label{eq:1st-error-term-in-ai-bj1}
    \tilde \mu ( a\tau^{+} > t, b\tau^{-} > t )  + 
    \tilde \mu ( a\tau^{+} + b \tau^{-} > t, \max\{a\tau^{+}, b\tau^{-}\} \leq t )
    = o( t^{-\gamma -1/\beta})
  \end{equation}
 and 
    \begin{equation}\label{eq:error-termin-ai-minus-bj}
    \tilde \mu ( a\tau^{+} > t , a\tau^{+} - b \tau^{-} \leq t )  = o(   t^{ - \gamma -  1/\beta}) 
  \end{equation}
\end{prop}
As we shall see, \eqref{eq:error-termin-ai-minus-bj} actually holds for all 
\( \gamma \in (0,1/\beta)\) (where \( 1/\beta > 1 \) since \( \beta \in (0,1)\)  by assumption) but we will not need this stronger statement. 
We prove Proposition \ref{lem:dist-of-atau-and-btau} in Section~\ref{sec:tau+tau-} and Proposition \ref{prop:higher-order} in Section \ref{sec:higher-order},  but first we show how they imply   Proposition~\ref{prop:tail-of-tau}. 

\begin{proof}[Proof of Proposition \ref{prop:tail-of-tau}]
To prove  \eqref{eq:a+b-tail}, first suppose that least one of  \( \ell_1, \ell_2 \) is non-zero. Substituting the corresponding lines of \eqref{eq:dist-of-tau+} and \eqref{eq:dist-of-tau-} into \eqref{eq:tail-of-Phi-positive-case}  and substituting \eqref{eq:1st-error-term-in-ai-bj1} into \eqref{eq:tail-of-Phi-mixed-case1} we obtain \eqref{eq:a+b-tail} in this case.
  If \( \ell_1 = \ell_2 = 0 \) we only need to establish an upper bound for \( \tilde \mu ( a \tau^+ + b \tau^- > t) \) rather than an asymptotic equality and therefore, instead of the decomposition in \eqref{eq:tail-of-Phi-positive-case1} and \eqref{eq:tail-of-Phi-positive-case}, we can use the fact that
  \begin{equation}\label{eq:tau-ineq-1}
    \tilde \mu ( a \tau^+ + b \tau^- > t ) \leq \tilde \mu ( a \tau^+ > t ) + \tilde \mu ( b \tau^- > t )
  \end{equation}
  The result then follows by inserting the corrsponding lines of \eqref{eq:dist-of-tau+} and \eqref{eq:dist-of-tau-} into \eqref{eq:tau-ineq-1}.

To prove \eqref{eq:a-b-pstve-tail}, if \( \ell_2 > 0 \) the result follows by substituting the corresponding line of \eqref{eq:dist-of-tau+} into \eqref{eq:tail-of-Phi-mixed-case} and  substituting \eqref{eq:error-termin-ai-minus-bj} into \eqref{eq:tail-of-Phi-mixed-case1}.
 Again, if  \( \ell_2 = 0 \) we only need to establish an upper bound for \( \tilde\mu ( a \tau^+ - b \tau^- > t) \) rather than an asymptotic equality and therefore, instead of the decomposition in \eqref{eq:tail-of-Phi-mixed-case} and \eqref{eq:tail-of-Phi-mixed-case1}, we can use the fact that
  \begin{equation}\label{eq:tau-ineq-2}
    \tilde \mu ( a \tau^+ - b \tau^- > t ) \leq \tilde \mu ( a \tau^+ > t )
  \end{equation}
  The result then follows by inserting the corresponding line of \eqref{eq:dist-of-tau+}  into \eqref{eq:tau-ineq-2}.
\end{proof}

\subsubsection{Leading order asymptotics}
\label{sec:tau+tau-}

We prove  Proposition \ref{lem:dist-of-atau-and-btau} via two lemmas which show in particular how the values \( \tilde h(0^-) ,  \tilde h(0^+)\) of the density of the measure \( \tilde \mu \) turn up in the constants \( C_a,  C_b\) defined in  \eqref{eq:def-of-Ca-Cb}. Our first lemma shows that the tails of the distributions \(   \tilde \mu( \tau^{+} > t )   \) and \( \tilde \mu( \tau^{-} > t )  \) have a very geometric interpretation. 

\begin{lem}\label{lem:distribution}
For every \( t > 0 \) we have
  \begin{equation}\label{eq:distribution1}
\tilde \mu( \tau^{+} > t ) 
    = \tilde \mu (y_{\lceil t \rceil}^-, 0)  \quand 
      \tilde \mu( \tau^{-} > t )  
    =   \tilde \mu (0, y_{\lceil t \rceil}^+ ) .
  \end{equation}
\end{lem}

\begin{rem}\label{rem:distribution}
While the first statement in \eqref{eq:distribution1} is  relatively straightforward, the second statement is not at all obvious since  \(\tau^-\) is defined on \( \Delta_0^-\) and there is no immediate connection  with the interval  \( (0, y_{\lceil t \rceil}^+) \) in \( \Delta_0^+\).  As we shall see, the proof of Lemma \ref{lem:distribution} requires  a  subtle and interesting argument. 
\end{rem}

\begin{rem}
Since \( \tilde \mu \) is equivalent to Lebesgue measure on \( \Delta_0^-\) and \( \Delta_0^+\),  we immediately have that \( \tilde \mu ( y_{\lceil t \rceil}^- , 0)  \approx |y_{\lceil t \rceil}^-|\) and \( \tilde \mu (0, y_{\lceil t \rceil}^+ )  \approx y_{\lceil t \rceil}^+\), and we can then use \eqref{eq:yn1} and \eqref{eq:yn2}, and Lemma \ref{lem:distribution},  to get upper bounds for the distributions \( \tilde \mu( \tau^{+} > t )  \) and \( \tilde \mu( \tau^{+} > t )  \). This is however not enough for our purposes as we  require sharper estimates for the distributions, and we therefore need a  more sophisticated argument which yields the statement in the following lemma. 
\end{rem}

\begin{lem} \label{lem:lips}
For every \( t > 0 \) we have 
  \[
 \tilde \mu ( y_{\lceil t \rceil}^-, 0 ) = y_{\lceil t \rceil}^- \tilde h( 0^-) + O( (y_{\lceil t \rceil}^-)^2)
 \quand  \tilde \mu ( y_{\lceil t \rceil}^+, 0 ) = y_{\lceil t \rceil}^+ \tilde h( 0^+) + O( (y_{\lceil t \rceil}^+)^2).
  \]
\end{lem}
Before proving these two lemmas we show how they imply  Proposition \ref{lem:dist-of-atau-and-btau}. 

\begin{proof}[Proof of Proposition \ref{lem:dist-of-atau-and-btau}]
  Let us first show \eqref{eq:dist-of-tau+}. Recall from the definition of \( C_a \) in \eqref{eq:def-of-Ca-Cb} that \( a = 0 \Rightarrow C_a = 0 \), so if \( a = 0 \) there is nothing to prove. Let us suppose then that \( a > 0 \). By Lemmas \ref{lem:distribution} and \ref{lem:lips} we have 
  \[
  \tilde \mu( a\tau^{+} > t ) = 
  \tilde \mu( \tau^{+} > t/a ) =
  y_{\lceil t/a \rceil}^- \tilde h( 0^-) + O( (y_{\lceil t/a \rceil}^-)^2).
  \]
  Then, using the asymptotic estimates  \eqref{eq:yn-0} and  \eqref{eq:yn1}  for \( y_{n}^- \) in Proposition \ref{prop:y_n-and-delta_n}; and since \( O ( t^{-2/\beta_1} ) = o ( t^{ - \gamma - 1 / \beta_1 } )\) for every \( \gamma \in ( 0, 1 ) \); and by the definition of \( C_a \) in \eqref{eq:def-of-Ca-Cb},   we obtain 
  \[
  \mu ( a\tau^+ > t ) = 
  \begin{cases} 
  B_1 \tilde h (0^-)  (t / a)^{-1/\beta_2} 
  + O ( t^{-2/\beta_2}) = C_a t^{-1/\beta_2} + o( t^{-\gamma - 1/\beta_2} ) & \text{if } \ell_1 > 0  \\
  O \left( ( 1 + b_2 )^{-t/ak_1} \right)& \text{if } \ell_1 = 0
  \end{cases}
  \]
  yielding \eqref{eq:dist-of-tau+}.
  To show \eqref{eq:dist-of-tau-} we can proceed similarly to the above. As before, if \( b = 0 \) there is nothing to prove so we assume \( b > 0\) in which case 
  Lemmas \ref{lem:distribution} and \ref{lem:lips} we have 
  \(
      \tilde \mu( b\tau^{-} > t ) = 
      \tilde \mu( \tau^{-} > t/b ) =
      y_{\lceil t/b \rceil}^+ \tilde h( 0^-) + O( (y_{\lceil t/b \rceil}^+)^2)
  \).
Now using \eqref{eq:yn+0} and \eqref{eq:yn2}, and arguing as above we find that \eqref{eq:a-b-neg-tail} holds for every \( \gamma \in ( 0, 1) \).
\end{proof}
  
We complete this section with the proofs of Lemmas \ref{lem:distribution} and \ref{lem:lips}.

\begin{proof}[Proof of Lemma \ref{lem:distribution}]
By definition, recall  \eqref{eq:taupm},  \( \tau^{+}(x)=i, \tau^{-}(x)=j\) for all  \( x\in  \delta_{i,j}\),  and therefore 
    \begin{equation}\label{eq:sumij1}
   \tilde \mu ( \tau^{+} > t) = 
    \sum_{i > t } \sum_{j = 1}^{\infty} \tilde \mu(\delta_{i,j}) 
    \quand 
        \tilde \mu ( \tau^{-} > t ) = 
    \sum_{j > t } \sum_{i = 1}^{\infty} \tilde \mu(\delta_{i,j}). 
 \end{equation}
We claim that for every \( i, j \geq  1 \) we have 
\begin{equation}\label{eq:sumijsub}
  \sum_{j = 1}^{\infty}  \tilde \mu (\delta_{i,j}) =  \tilde \mu(\delta_i^-)
\quand
\sum_{i = 1}^{\infty}  \tilde \mu (\delta_{i,j}) =  \tilde \mu(\delta_j^+). 
\end{equation}
Then,  substituting \eqref{eq:sumijsub} into \eqref{eq:sumij1} we get 
    \begin{equation} 
   \tilde \mu ( \tau^{+} > t) = 
   \sum_{i > t} \sum_{j = 1}^{\infty} \tilde \mu( \delta_{i,j} ) = \sum_{i > t} \tilde \mu ( \delta_{i}^{-} ) = \tilde \mu ( y_{\lceil t \rceil}^-, 0 ), 
 \end{equation}
 and 
  \[
    \tilde \mu ( \tau^{-} > t ) = 
    \sum_{j > t } \sum_{i = 1}^{\infty} \tilde \mu(\delta_{i,j}) = \sum_{j > t} \tilde \mu ( \delta_{i}^{+} ) = \tilde \mu ( 0, y_{\lceil t \rceil}^+ )
  \]
  which is exactly the  statement \eqref{eq:distribution1} in the Lemma.
  
Thus it only remains to prove \eqref{eq:sumijsub}. As already mentioned in Remark \ref{rem:distribution}, despite the apparent symmetry between the two statements, the situation in the two expressions is actually quite different. Indeed,  from the topological construction of the induced map, for each  \( i \geq 1 \) we have  
\begin{equation}\label{eq:unionij1}
\delta_{i}^{-}= \bigcup_{j=1}^{\infty} \delta_{i,j}
 \end{equation}
which, since the intervals \( \delta_{i,j} \) are pairwise disjoint,  clearly implies the first equality in \eqref{eq:sumijsub}. The second equality is not immediate since, for each fixed \( j \geq 1\), ,  the intervals \( \delta_{i,j}\)  are \emph{spread out} in \( \Delta_0^-\), with each \( \delta_{i,j}\) lying inside the corresponding interval \( \delta_i^-\), and indeed the \( \delta_{i,j}\) do not even belong to \( \delta_j^+\) and therefore we cannot just substitute \( i \) and \( j \) to get a corresponding version of  \eqref{eq:unionij1}.
We use instead a simple but clever argument inspired by a similar argument in \cite[Lemma 8]{CriHayMarVai10} which takes advantage of the invariance of the meaure \( \tilde \mu \).  Recall first of all from the construction of the induced map, that \(  g^{-1} (\delta_{ j }^+) \) consists of exactly two  connected components,  one is exactly the interval \( \delta_{1,j}  \)  and the other one is a subinterval of \( \Delta_1^+\). So for any \( j \geq 1 \) we have
 \[
 g^{-1} (\delta_{ j }^+) = \delta_{1,j} \cup \{ x : \Delta_1^+ : g(x) \in \delta_j^{+}\}.
  \]
  By the invariance of the measure \( \tilde \mu \), and since these two components are disjoint,  this implies 
  \begin{equation}\label{eq:sumij}
  \tilde \mu(\delta_{ j }^+) = \tilde \mu(g^{-1} (\delta_{ j }^+)) = \tilde \mu(\delta_{1,j}) + \tilde \mu( \{ x : \Delta_1^+ : g(x) \in \delta_j^{+}\})
  \end{equation}
The preimage of the set  \( \{ x : \Delta_1^+ : g(x) \in \delta_j^{+}\}\)  itself also has two disjoint connected components 
 \[
 g^{-1} \{ x \in \Delta_1^+ : g(x) \in \delta_j^{+} \} = \delta_{2,j} \cup \{ x \in \Delta_2^+ : g^2(x) \in \delta_j^{+} \}
 \] 
 and therefore, again by the invariance of \( \tilde \mu \), we get 
  \[
 \tilde \mu(g^{-1} \{ x \in \Delta_1^+ : g(x) \in \delta_j^{+} \}) = \tilde \mu(\delta_{2,j}) + \tilde \mu(\{ x \in \Delta_2^+ : g^2(x) \in \delta_j^{+} \})
 \] 
 and, substituting this into \eqref{eq:sumij}, we get
 \[
  \tilde \mu(\delta_{ j }^+) = \tilde \mu(g^{-1} (\delta_{ j }^+)) = \tilde \mu(\delta_{1,j}) + \tilde \mu(\delta_{2,j}) + \tilde \mu(\{ x \in \Delta_2^+ : g^2(x) \in \delta_j^{+} \}).
 \]
 Repeating this procedure \( n \) times gives 
    \[
  \tilde \mu(\delta_{ j }^+) = \sum_{i = 1}^{n} \tilde \mu(\delta_{i,j}) + \tilde \mu(\{ x \in \Delta_n^+ : g^n(x) \in \delta_j^{+} \})
 \]
 and therefore inductively, we obtain \eqref{eq:sumijsub}, thus completing the proof. 
\end{proof}

\begin{proof}[Proof of Lemma \ref{lem:lips}]

  From Lemma \ref{lem:distribution}  we can give precise estimates for \( \tilde \mu ( \tau^{\pm} > t ) \) in terms of the \( y_{\lceil t \rceil}\) by making use of the fact that \( \tilde h \) is Lipschitz on \( \Delta_0^{\pm} \) (see Corollary \ref{cor:density}). Indeed,
  \[
  \tilde \mu ( \tau^- > t ) = \tilde \mu ( 0, y_{\lceil t \rceil}^+) =  \int_0^{y_{\lceil t \rceil}^+} \tilde h (x) dx = y_{\lceil t \rceil}^+ \tilde h( 0^+)
  + \int_0^{y_{\lceil t \rceil}^+} \tilde h (x) - \tilde h (0^+) dx.
  \] 
  Using the fact that the density is Lipschitz we have
  \[
  \left |\int_0^{y_{\lceil t \rceil}^+} \tilde h (x) - \tilde h (0^+) dx \right|
  \lesssim \int_0^{y_{\lceil t \rceil}^+}x dx  \lesssim (y_{\lceil t \rceil}^+)^2
  \]
  and so 
  \(
  \tilde \mu ( \tau^- > t ) = y_{\lceil t \rceil}^+ \tilde h( 0^+) + O( (y_{\lceil t \rceil}^+)^2).
  \)
  The statement for \( \mu ( \tau^+ > t ) \) follows in   the same way.
  \end{proof}

\subsubsection{Higher order asymptotics}
\label{sec:higher-order}

In this subsection we prove Proposition \ref{prop:higher-order}. For clarity we prove \eqref{eq:1st-error-term-in-ai-bj1} and \eqref{eq:error-termin-ai-minus-bj} in two separate lemmas. We will make repeated use of some upper bounds for the measure \( \tilde \mu ( \delta_{i,j} ) \) of the partition elements which are given in Lemma \ref{lem:estimates-on-delta-ij}

\begin{lem} 
\label{lem:error-a+b-+ve}
  If at least one of \( \ell_1, \ell_2 \) are not zero, then for every \( a, b \geq 0 \)
  \begin{equation}
    \label{eq:1st-error-term-in-ai-bj1a}
    \tilde \mu ( a\tau^{+} > t, b\tau^{-} > t )  + 
    \tilde \mu ( a\tau^{+} + b \tau^{-} > t, \max\{ a\tau^{+}, b\tau^{-}\} \leq t )  
    = o( t^{-\gamma -1/\beta})
  \end{equation}
  for any \( \gamma \in ( 0, 1 ) \).
\end{lem}

\begin{proof}
  First note that if one of \( a,b \) is \( 0 \) then \eqref{eq:1st-error-term-in-ai-bj1a} is automatically satisfied.

  Now suppose that \( a, b > 0 \).
  For the first  term in \eqref{eq:1st-error-term-in-ai-bj1a}, from Lemma \ref{lem:estimates-on-delta-ij} we get 
  \begin{equation}
    \label{eq:1st-error-term-in-ai-bj}
    \tilde \mu ( a\tau^{+} > t, b\tau^{-} > t )
    = \sum_{ i = t/a }^{\infty}\sum_{j = t/b }^{\infty} \tilde \mu (\delta_{i,j}) 
    \lesssim \sum_{i = t/a }^{\infty}\sum_{j = t/b }^{\infty} (ij)^{-(1 + 1/\beta)}
    \lesssim t^{-2/\beta}
  \end{equation}
  which is \( o( t^{-\gamma -1/\beta}) \) for every \( \gamma \in (0, 1/\beta) \) and therefore in particular for every \( \gamma \in (0, 1) \).  
  
For the second  term in \eqref{eq:1st-error-term-in-ai-bj1a} we obtain  from Lemma \ref{lem:estimates-on-delta-ij} that 
  \begin{align}
    \label{eq:the-horrible-sum}
    \tilde \mu ( a\tau^{+} + b \tau^{-} > t, \max\{ a\tau^{+}, b\tau^{-}\} \leq t )  
    &= \sum_{ i = 1 }^{t/a} \sum_{j = \frac{t - ai + 1}{b}}^{t/b} \tilde \mu(\delta_{i,j}) \lesssim
    \sum_{ i = 1 }^{t/a} \sum_{ j = \frac{ t - ai + 1}{b}}^{t/b} 
    i^{-1 - 1/\beta} j ^{-1 - 1/\beta} \\
    \nonumber
    &\lesssim
    \sum_{ i = 1 }^{t/a} i^{-1-1/\beta} \left[ 
    \left( \frac{t - ai + 1}{b} \right)^{-1/\beta} - \left(\frac{t}{b}\right)^{-1/\beta}
    \right] \\
    \nonumber
    &\lesssim
    t^{-1/\beta} \sum_{ i = 1 }^{ t / a } i^{- 1 - 1/\beta } \left[ 
    \left( 1 - \frac{ ai - 1 }{t} \right)^{1/\beta} - 1
    \right]
  \end{align}
      Making the change of variables \( k = \lceil ai - 1 \rceil \) and using  that the first term in the sum is \( 0 \) we obtain
      \[
        \tilde \mu ( a\tau^{+} + b \tau^{-} > t, \max\{a\tau^{+}, b\tau^{-}\} \leq t )  
        \lesssim  
        t^{-1/\beta} \sum_{ k = 1 }^{ t - 1 } k^{- 1 - 1/\beta } \left[ 
          \left( 1 + \frac{k}{t} \right)^{1/\beta} - 1
        \right].
      \]
      Let us set \( a_k (t) \coloneqq  k^{- 1 - 1/\beta } \left[\left( 1 + \frac{k}{t} \right)^{1/\beta} - 1 \right] \) and use the binomial theorem to get      
\[
        a_{k}(t) 
        = \frac{1}{k^{1+1/\beta}} \sum_{ m = 1 }^{ \infty } \binom{ - 1/\beta }{ m } \left( - \frac{ k } { t } \right)^m 
        = \frac{1}{t k^{1/\beta}}    \sum_{ m = 1 }^{ \infty } 
          \binom{ - 1/\beta }{ m - 1 } 
          \left( \frac{1}{m} \left( \frac{1}{\beta} - 1 \right) + 1 \right)
          \left( - \frac{ k } { t } \right)^{ m - 1 }. \\
\]
As  \(   \left( \frac{1}{\beta} - 1 \right)/m \) is uniformly bounded above by some constant depending only on \( \beta \) we obtain
      \[ 
        a_{k} (t) \lesssim k^{-1/\beta} t^{ -1 } \sum_{ k = 0 }^{\infty} \binom{ - 1/\beta }{ m - 1 } \left( - \frac{ k } { t } \right)^{ m - 1 } = k^{-1/\beta} t^{-1} \left( 1 - \frac{k}{t} \right)^{-1/\beta}.
      \]
      Using the fact that \( n/( n - 1) < 2 \) and that \( 1/\beta > 1 \) we may conclude
      \begin{align*}
        \tilde \mu ( a\tau^{+} + b \tau^{-} > t, \max\{a \tau^{+}, b\tau^{-}\} \leq t )  
        &\lesssim  
        t^{- 1 -1/\beta} \sum_{ k = 1 }^{ t - 1 } \left( \frac{ t }{ k ( t - k ) }\right)^{1/\beta} 
        \lesssim
        t^{- 1 -1/\beta} \sum_{ k = 1 }^{ t - 1 } \frac{ t }{ k ( t - k ) } \\
        &\lesssim
        t^{- 1 -1/\beta} \int_{ 1 }^{ t - 1 } \frac{ t }{ x ( t - x ) } dx \lesssim 
        t^{-1 -1/\beta} \log(t) = o( t^{-\gamma -1/\beta})
      \end{align*}
      for any \( \gamma \in ( 0, 1 ) \). 
   \end{proof}


\begin{lem}
\label{lem:error-term-a-b}
  If \( \ell_1, \ell_2 \) are not both zero, then for every \( a , b \geq 0 \), \( \gamma\in (0,1/\beta)\) we have
  \begin{equation*}
    \tilde \mu ( a\tau^{+} > t , a\tau^{+} - b \tau^{-} \leq t )  = o(   t^{ - \gamma -  1/\beta}).
  \end{equation*}
\end{lem}

\begin{proof}
By Lemma \ref{lem:estimates-on-delta-ij}  we get 
      \begin{align}
        \nonumber
        \tilde \mu ( a\tau^{+} > t , a\tau^{+} - b \tau^{-} \leq t )
        & = \sum_{ i > t/a } \tilde \mu ( \tau^{+} = i, b\tau^{-} \geq ai - t ) 
         = \sum_{ i > t/a } \sum_{ j \geq (ai - t)/b } \tilde \mu( \delta_{i, j}) \\
        \nonumber
        & \lesssim \sum_{ i > t/a } i^{- (1 + 1/\beta) }(ai - t )^{-1/\beta}  
        \lesssim \sum_{ i = 1}^{\infty} ( i + t )^{-(1 + 1/\beta)}i^{-1/\beta}.
      \end{align}
      We claim that 
      \[ \sum_{ i = 1}^{ \infty } ( t + i )^{ -  1 - 1/\beta  }i^{ -1/\beta }
      \lesssim  t^{- \gamma - 1/\beta  } \]
       for every \( 0 < \gamma < 1/\beta \), which is equivalent to showing that 
      \[ 
      \sum_{ i = 1}\frac{t^{\gamma + 1/\beta}}{(t + i)^{1 + 1/ \beta }i^{ 1 / \beta }} \leq C \]
       for some \( C > 0 \) independent of \( t \).
      Indeed, for every \( i \),
      \begin{align*}
        \frac{t^{\gamma + 1/\beta}}{(t + i)^{1 + 1/ \beta }i^{ 1 / \beta }}
        \leq \frac{ ( t + i )^{ \gamma + 1 / \beta} }{ (t + i)^{1 + 1 / \beta } i^{1/\beta} }
        = \frac{ 1 }{ ( t + i )^{ 1 - \gamma } i^{1/\beta} }
        = \frac{ 1 }{ ( t/i + 1 )^{ 1 - \gamma } i^{  1 - \gamma + 1/\beta} }
        \leq \frac{1}{ i^{ 1 - \gamma + 1/\beta} }
      \end{align*}
      which is summable for every \( 0 < \gamma < 1/\beta \). This implies the claim and thus the lemma. 
\end{proof}

\subsection{Estimates for the induced observables}
\label{sec:proof-phi-in-Lq}

In this section we prove Corollary \ref{cor:tail-tau} and Proposition \ref{prop:bounds-on-tilde-Phi}. We recall (see paragraph at the beginning of Section \ref{sec:distind}) that  we will only explicitly treat the case that \( \ell_{1}, \ell_{2} > 0 \) (and thus in particular \( \beta > 0 \)). Throughout this section we will assume that \( \varphi \) is a H\"older obervable and define \( a = \varphi (1) \), \( b = \varphi (-1) \).

\subsubsection{Proof of Corollary \ref{cor:tail-tau}}

We first consider the case where  \( \beta_{\varphi} = \beta \). Recall from \eqref{eq:def-B-phi} that \( \beta_{\varphi} = \beta \) occurs when \( \varphi \) is non-zero at a fixed point corresponding to the maximum of \(  \beta_{1}, \beta_{2} \); and that \( \beta_{\varphi} \neq \beta \) occurs when \( \beta_{1} \neq \beta_{2}\) and \( \varphi \) is zero at the fixed point corresponding to the minimum of \(  \beta_{1}, \beta_{2} \).

\begin{lem}\label{lem:cor}
If  \( \beta_{\varphi} = \beta \), then \( \tau_{a,b} \in \mathcal{D}_{1/\beta_{\varphi}} \). In particular, if  \( \beta_{\varphi} \in (0,1/2) \) then \( \tau_{a,b} \in L^2( \hat \mu ) \). 
\end{lem}

\begin{proof}
Notice first of all that if  \( \tau_{a,b} \in \mathcal{D}_{1/\beta_\varphi} \) then,  in particular, \( \tilde \mu ( \pm \tau_{a,b} > t ) \lesssim t^{-1/\beta_{\varphi}} \) and so, if \( \beta_{\varphi} \in (0,1/2) \) we obtain that \( \tau_{a,b} \in L^2 ( \hat \mu ) \).
  Thus we just need to prove that \( \tau_{a,b} \in \mathcal{D}_{1/\beta_\varphi} \).

  Suppose first that  \( a, b \) do not have opposite signs, i.e. either \( a,b \geq 0 \) or \( a,b \leq 0 \), in which  case the distribution of \( \tau_{a,b} \) is determined by the first case of \eqref{eq:a+b-tail}. Therefore, taking \( \gamma = 0 \) we get 
  \[
   \tilde \mu ( \tau_{a,b} > t )  = C_a t^{-1/\beta_2} + C_b t^{-1/\beta_1} + o ( t^{- 1/\beta } ) 
    = c t^{-1/\beta} + o ( t^{-1/\beta})
  \]
  for some constant \( c > 0 \), which may be equal to \( C_{a} \), 
  \( C_{b}\), or \( C_{a}+C_{b},\) depending on the relative values of 
  \( \beta_{1}, \beta_{2}\). If \( a,b\geq 0 \), and exactly the same tail for
   \( \tilde  \mu ( \tau_{a,b} < - t) \) if \( a, b \leq 0 \). 
 By \eqref{eq:def-D-alpha} and the fact that \( \beta_{\varphi} = \beta\) we get 
  that  \( \tau_{a,b} \in \mathcal{D}_{1/\beta_\varphi} \), thus proving the result in this case. 
  If \( a \geq 0 \), \( b \leq 0 \)  the distribution of \( \tau_{a,b} \) is given by \eqref{eq:a-b-pstve-tail} and \eqref{eq:a-b-neg-tail} and so, taking \( \gamma = 0 \) gives 
    \[
    \tilde \mu ( \tau_{a,b} > t ) = C_a t^{-1/\beta_2}  + o (t^{-1/\beta}) = c_1 t^{-1/\beta} + o(t^{-1/\beta}) 
  \]
  and 
  \[
   \tilde  \mu ( \tau_{a,b} < - t ) = C_{|b|} t^{-1/\beta_1}  + o (t^{-1/\beta}) = c_2 t^{-1/\beta} + o(t^{-1/\beta}) 
  \]
  where \( c_{1} = C_{a}\) and \( c_{2}= C_{|b|}\) if \( \beta_2 = \beta\) and \( \beta_1 = \beta \) respectively, and equal to 0 otherwise.
  At least one of the \( c_1, c_2 \) has to be non-zero as \( \beta_{\varphi} = \beta \) implies that \( \varphi \) is non-zero at a fixed point corresponding to the largest of \( \beta_1, \beta_2 \) and so if \( \beta_1 = \max \{ \beta_1, \beta_2 \} \) we know from \eqref{eq:def-of-Ca-Cb} that \( c_2 = C_{|b|} > 0 \) and if \( \beta_2 = \max \{ \beta_1, \beta_2 \}  \) we know from \eqref{eq:def-of-Ca-Cb} that \( c_1 = C_a > 0 \). Thus, since \( \beta=\beta_{\varphi}\), we get \( \tau_{a,b} \in \mathcal{D}_{1/\beta_{\varphi}} \). If   \( a \leq 0 \), \( b \geq 0 \) the same argument holds exchanging  the roles of the positive and negative tails. 
\end{proof}


\begin{proof}[Proof of Corollary \ref{cor:tail-tau}]
We have already proved the result for \( \beta_{\varphi}=\beta\) in Lemma \ref{lem:cor} so we can assume  that \( \beta_{\varphi} \neq \beta \).  This implies that  \( \beta_1 \neq \beta_2 \) and  that \( \varphi \) is only non-zero at the fixed point corresponding to the smallest of the \( \beta_1, \beta_2 \). This situation can arrise in two ways: either (i) \( a \neq 0 \),  \( b = 0 \) and  \( \beta_{\varphi} = \beta_2 < \beta_1 \); or (ii) \( a = 0 \), \( b \neq 0 \) and \( \beta_\varphi = \beta_1 < \beta_2 \). We will assume (i) and give an explcit proof of the Lemma. The proof of the Lemma in situation (ii) then follows in the same way.

  Under our assumptions we know from Proposition \ref{prop:tail-of-tau} that the tail of \( \tau_{a,b} \) is determined by \eqref{eq:a+b-tail}, and we recall from \eqref{eq:def-of-Ca-Cb} that \( C_b = 0 \). If \( \beta_{\varphi} = \beta_2  = 0 \) then we know from the second line of \eqref{eq:a+b-tail} that 
  \begin{equation*}\label{eq:t}
   \hat \mu ( \pm \tau_{a,b} > t ) \lesssim t^{-\gamma -1 /\beta_1 }.
   \end{equation*}
    Since  \( \beta = \beta_1 < 1 \)  by assumption, we may choose  \( \gamma  \in [0,1 ) \) such that  \( \gamma + 1 / \beta_1 > 2 \) yielding \( \tau_{a,b} \in L^2( \hat \mu ) \). If \( \beta_{\varphi} \in (0,1/2) \) then the first line of \eqref{eq:a+b-tail} gives that 
   \[
    \hat \mu ( \pm \tau_{a,b} > t ) \lesssim \max \{ t^{-1/\beta_{\varphi}}, t^{-\gamma - 1/\beta} \} 
    \]
     for any \( \gamma \in [0,1) \). Choosing \( \gamma \) as before so that  \( \gamma + 1 /\beta_1 >  2 \) we again obtain that \( \tau_{a,b} \in L^2( \hat \mu ) \).
  If \( \beta_\varphi = \beta_2 \in [1/2,1) \) then, choosing \( \gamma \in [0,1) \) so that \( \gamma + 1 / \beta_1 > 1 / \beta_{2} \) we know from \eqref{eq:a+b-tail} that the non-zero tail of \( \tau_{a,b} \) is given by
  \[
     \hat \mu ( \pm \tau_{a,b} > t ) \lesssim   C_a t^{-1/\beta_2}  + o ( t^{-\gamma - 1/\beta_1 } ) 
    = C_a t^{-1/\beta_2} + o ( t^{-1/\beta_2}) 
  \]
  yielding  \( \tau_{a,b} \in \mathcal{D}_{1/\beta_\varphi} \).
\end{proof}

\subsubsection{Proof of Proposition \ref{prop:bounds-on-tilde-Phi}}

\begin{proof}[Proof of Proposition \ref{prop:bounds-on-tilde-Phi}]

  For a point \(x \in \delta_{i,j}^{-}\) we know that \(\tau(x) = i + j\) and that
  \begin{equation}
      \label{eq:orbit-of-delta-ij}
      g^{k}(x) \in \Delta_{i-k}^+  \quad  \forall \ 1 \leq k \leq i;\quad \text{ and } \quad
      g^{i + k}(x) \in \Delta_{j-k}^{-} \quad  \forall \ 1 \leq k \leq j - 1.
  \end{equation}
  Recall that by Proposition  \ref{prop:xn-and-Deltan} we have \(    1- x_n^{+}  \lesssim n^{- {1}/{\ell_2}},\)  and  \( |\Delta^{+}_{n}| \lesssim n^{-(1+1/\ell_{2})} \ll 1- x_n^{+}  \), which means that we can use the fact that   \( \tilde \varphi (1) = 0 \) and the H\"older continuity of \( \tilde \varphi_{(0,1]}\)  to obtain
  \begin{equation}
    \label{eq:bound-with-i}
    |\tilde \varphi \circ g^{k}(x)| \lesssim (1- x_{i-k}^{+})^{\nu_{2}}\leq (i-k)^{-\nu_{2}/\ell_{2}},
  \end{equation}
  for all \(  1 \leq k \leq i-1 \). Similarly, using the fact that  \( \tilde \varphi (-1) = 0 \) and the H\"older continuity of \( \tilde \varphi_{[-1,0)} \),
  \begin{equation}
    \label{eq:bound-with-j}
    |\tilde \varphi \circ g^{i+k}(x)| \lesssim (1+x_{j-k}^{+})^{\nu_{1}}\leq (j-k)^{-\nu_{1}/\ell_{1}},
  \end{equation}
  for all for \( 1 \leq k \leq j - 1\).
 For \( x \in \delta_{i,j} \) we know from \eqref{eq:bound-with-i} and \eqref{eq:bound-with-j} that
  \begin{equation}\label{eq:sum}
    |\tilde \Phi(x)|
    \lesssim \sum_{k=1}^{i-1} (i-k)^{-\frac{\nu_2}{\ell_2}} + 
    \sum_{k=1}^{j-1} (j-k)^{-\frac{\nu_1}{\ell_1}}. 
  \end{equation}
We now consider two cases. Suppose first that  \( \ell_1 < \nu_1 \) and \( \ell_2 < \nu_2 \).  Then \( |\tilde \Phi (x)| \) is uniformly bounded in \( x \), as both \eqref{eq:bound-with-i} and \eqref{eq:bound-with-j} are summable in \( k \) and therefore the sums in \eqref{eq:sum} both converge.  Therefore \( \tilde \Phi \in L^q ( \hat{\mu} ) \) for every \( q > 0 \), in particular \( \tilde \Phi \in L^2 ( \hat{\mu} ) \)
giving the first implication in \eqref{eq:bounds1}, and by Chebyshev's inequality,  \( \hat \mu ( \pm \tilde \Phi > t ) = O(t^{-q}) \) for every \( q > 0 \), giving the second implication in  \eqref{eq:bounds1}. Notice that we have not required in this case the conditions \ref{itm:H1} and \ref{itm:H2}. 

Now suppose that \( \ell_1 \geq  \nu_1 \) and/or \( \ell_2 \geq  \nu_2 \) and suppose also  that 
    \begin{equation}
    \label{eq:cond-on-nu}
  \nu_{1} > \frac{ \beta_{1} - 1/q }{ k_{2} } \quand
 \nu_{2} > \frac{ \beta_{2} - 1/q }{ k_{1} }.
  \end{equation}
  Notice that for \( q = 2 \) this gives exactly  condition \ref{itm:H1}  and for  \( q = 1/ \beta_{\varphi} \) this gives exactly  \ref{itm:H2}. We can also suppose without loss of generality that in fact \( \ell_1 > \nu_1 \) and/or \( \ell_2 > \nu_2 \) since we can  decrease slightly the H\"older exponent while still satisfying \eqref{eq:cond-on-nu}. 
In this case the sums in \eqref{eq:sum} diverge but  admit the following bounds: 
    \[
    |\tilde \Phi(x)|
    \lesssim \sum_{k=1}^{i-1} (i-k)^{-\frac{\nu_2}{\ell_2}} + 
    \sum_{k=1}^{j-1} (j-k)^{-\frac{\nu_1}{\ell_1}} 
    \lesssim  i^{1 -\frac{\nu_2}{\ell_2}} + j^{1 - \frac{\nu_1}{\ell_1}}.
  \]
We can then bound the integral by 
  \begin{align*}
      \int_{\Delta_0^-} |\tilde{\Phi} (x) |^q dm
      & \lesssim \left[
          \sum_{i = 1}^{\infty}\sum_{j=1}^{\infty}
          |\delta_{i,j}|  \left(
              i^{1 - \frac{\nu_2}{\ell_2}}
              + j^{1 - \frac{\nu_1}{\ell_1}}
          \right)^q
      \right] 
  \end{align*}
Then, since  \( | \delta_{i,j}| = O( i^{- ( 1 + 1/\beta_2 )} j^{ -( 1 +  1 / \beta_1 )}) \) we get 
  \begin{align*}
      \int_{\Delta_0^-} |\tilde{\Phi} (x) |^q dm
           & \lesssim  \left[
          \sum_{i,j = 1}^{\infty}
          \frac{
              i^{q - \frac{q\nu_2}{\ell_2}}
          }
          {
              i^{1 + \frac{1}{\beta_2}}
              j^{1 + \frac{1}{\beta_1}}
          }
          + \sum_{i,j=1}^{\infty}
          \frac{
              j^{q - \frac{q\nu_1}{\ell_1}}
          }
          {
              i^{1 + \frac{1}{\beta_2}}
              j^{1 + \frac{1}{\beta_1}}
          }
        \right] 
        \lesssim
          \sum_{i = 1}^{\infty}
              i^{q - \frac{q\nu_2}{\ell_2} - 1 - \frac{1}{\beta_2}}
          + \sum_{j=1}^{\infty}
              j^{q - \frac{q\nu_1}{\ell_1} - 1 - \frac{1}{\beta_1}}
  \end{align*}
The latter sums are bounded   exactly when \eqref{eq:cond-on-nu} holds. As mentioned above, for \( q= 2 \) this is exactly condition \ref{itm:H1} and therefore we get that \( \tilde \Phi \in L^{2} ( \hat \mu ) \). 
 For \( q = 1/\beta_{\varphi} \) this is exactly condition \ref{itm:H2} and therefore we get that \( \tilde \Phi \in L^{q} ( \hat \mu  )\). In fact if \eqref{eq:cond-on-nu} holds for \( q = 1/\beta_{\varphi} \) then  there exists some \( \eps > 0 \) such that \eqref{eq:cond-on-nu} holds for all \( q \in [1/\beta_{\varphi}, 1/\beta_{\varphi} + \eps)\) and therefore 
  \(  \tilde \Phi \in L^{q} ( \hat \mu  )\) for every \( q \in [1/\beta_{\varphi}, 1/\beta_{\varphi} + \eps)\). From this and   Chebyschev's inequality we get \( \hat \mu ( \pm \tilde \Phi > t ) = o ( t^{ - 1 / \beta_{\varphi} } )\).
\end{proof}
\medskip

\medskip

\noindent
\textbf{Acknowledgements}
\\
We would like to thank Emanuel Carneiro, Sylvain Crovisier, and Santiago Martinchich
for their comments and suggestions regarding early versions of this work.
\\
\\
\noindent
D.C. was partially supported by the ERC project 692925 \emph{NUHGD} and by the Abdus Salam ICTP visitors program.
\\
\noindent
The authors declare that there are no other competing interests relevant to the research communicated in this paper.

\printbibliography
\end{document}